\documentclass[a4paper]{article}

\usepackage{amsmath}
\usepackage{amssymb}
\usepackage{amsthm}
\usepackage{mathrsfs}
\usepackage{thmtools}
\usepackage{wasysym}
\usepackage{bbm}		%
\usepackage{mathtools}
\usepackage[shortlabels]{enumitem}
\usepackage{braket}		%
\usepackage{mdwlist}		%
\usepackage{xcolor} 
\usepackage{tikz}
\usepackage{graphicx}
\usepackage[clock]{ifsym}
\usepackage[a4paper, left=2.7cm, right=2.7cm, top=3.5cm, bottom=2cm]{geometry}
\usepackage{cleveref}
\usepackage{fontawesome}
\usetikzlibrary{patterns,positioning,arrows,decorations.markings,calc,decorations.pathmorphing,decorations.pathreplacing}
\theoremstyle{plain}
\newtheorem{theorem}{Theorem}[section]
\newtheorem{lemma}[theorem]{Lemma}
\newtheorem{corollary}[theorem]{Corollary}
\newtheorem{proposition}[theorem]{Proposition}
\newtheorem{definition}[theorem]{Definition}

\newtheorem{question}[theorem]{Question}
\newtheorem{conjecture}[theorem]{Conjecture}

\theoremstyle{definition}
\newtheorem{remark}[theorem]{Remark}
\newtheorem*{acknowledgements*}{Acknowledgements}
\newtheorem{example}[theorem]{Example}

\newcommand{\exampleqed}{\ensuremath{\Diamond}\par}
\newcommand{\cuadri}{\square^{\mathbb{Z}^2}}

\newcommand{\ZZ}{\mathbb{Z}}			%
\newcommand{\NN}{\mathbb{N}}			%
\newcommand{\RR}{\mathbb{R}}			%

\newcommand{\FF}{\mathcal{F}}
\newcommand{\C}{\mathcal{C}}

\newcommand{\htop}{h_{\text{top}}}
\newcommand{\entsft}{\mathcal{E}_{\text{SFT}}}

\newcommand{\isdef}{\triangleq}			%

\newcommand{\supp}{%
	\operatorname{\mathrm{supp}}%
}

\newcommand{\ee}{\mathrm{e}}			%
\newcommand{\define}[1]{\textbf{#1}}

\title{%
	On the entropies of subshifts of finite type on countable amenable groups\\
}

\usepackage{authblk}

\author{Sebasti\'an Barbieri}
\affil{Universidad de Santiago de Chile}

\begin{document}
	
	\maketitle	
	
	\begin{abstract}
		Let $G,H$ be two countable amenable groups. We introduce the notion of group charts, which gives us a tool to embed an arbitrary $H$-subshift into a $G$-subshift. Using an entropy addition formula derived from this formalism we prove that whenever $H$ is finitely presented and admits a subshift of finite type (SFT) on which $H$ acts freely, then the set of real numbers attained as topological entropies of $H$-SFTs is contained in the set of topological entropies of $G$-SFTs modulo an arbitrarily small additive constant for any finitely generated group $G$ which admits a translation-like action of $H$. In particular, we show that the set of topological entropies of $G$-SFTs on any such group which has decidable word problem and admits a translation-like action of $\ZZ^2$ coincides with the set of non-negative upper semi-computable real numbers. We use this result to give a complete characterization of the entropies of SFTs in several classes of groups. 
		
		\textbf{Corrigendum}: An error has been found in the proof of Theorem 4.7. We have added a corrigendum appendix which explains the error, discusses possible solutions and details which results from Section 5 still hold (the only result that is no longer proven is Corollary 5.12). We also provide an update on the state of the art concerning the questions asked in Section 6. 
	\end{abstract}
	
	\medskip	
	
	\noindent
	\textbf{Key words and phrases:} topological entropy, symbolic dynamics, subshifts of finite type, amenable groups, cocycles of group actions.
	\smallskip
	
	\noindent
	\textbf{MSC2010:} \textit{Primary:}
	37B40, %
	37B10. %
	\textit{Secondary:}
	22F05, %
	37B05.  %
	
\section{Introduction}

	The topological entropy of an action $G \curvearrowright X$ of an amenable group $G$ on a compact metric space $X$ by homeomorphisms is a non-negative number which counts the asymptotic exponential growth rate of the number of distinguishable orbits of the system. Initially introduced by Adler, Konheim and McAndrew~\cite{AdlerKonheimMcAndrew1965} for $\ZZ$-actions, it is an important conjugacy invariant which has been studied broadly.
	
	A particularly interesting case is when $G\curvearrowright X$ is a subshift of finite type ($G$-SFT). Up to dynamical conjugacy, there are countably many distinct subshifts of finite type, and therefore at most countably many real numbers can be attained as the entropy of a subshift of finite type. A classical result by Lind~\cite{Lind1984} classifies the topological entropies attainable by $\ZZ$-SFTs as non-negative rational multiples of logarithms of Perron numbers. This characterization relies on a full description of the configurations of $\ZZ$-SFTs as bi-infinite paths on a finite graph and a study of the eigenvalues of their adjacency matrices.
	
	A more recent result by Hochman and Meyerovitch~\cite{HochmanMeyerovitch2010} completely classifies the entropies of $\ZZ^d$-SFTs. Interestingly, they show that for $d \geq 2$ the characterization is of an algorithmic nature. More precisely, the numbers attained as entropies of $\ZZ^d$-SFTs coincides with the set of non-negative upper semi-computable real numbers. Their classification relies on a construction which embeds arbitrarily large computation diagrams of an arbitrary Turing machine into a $\ZZ^d$-SFT. 
	
	The purpose of this study is to explore what entropies can be achieved by subshifts of finite type defined on an arbitrary amenable group $G$. In particular, we shall present a way to transfer entropies attainable by SFTs on a group $H$ to $G$ whenever $H$ can be ``geometrically embedded into $G$". A simple observation is that whenever $H$ is a subgroup of an amenable group $G$, then any number obtained as the topological entropy of an $H$-SFT $X$ can also be obtained as the topological entropy of a $G$-SFT $Y$. Indeed, this is achieved by letting $Y$ be the set of all configurations such that every $H$-coset contains a configuration of $X$ and there are no restrictions between each individual $H$-coset. 

	In this article we generalize the above construction introducing the notion of group charts. A group chart $(X,\gamma)$ is a dynamical structure consisting of a dynamical system $G \curvearrowright X$ and a continuous cocycle $\gamma\colon H \times X \to G$ that associates configurations in $X$ to partitions of its underlying group $G$ into quotients of $H$. Whenever $X$ is a $G$-subshift, we can use the partitions induced by the chart $(X,\gamma)$ to embed any $H$-subshift $Y$ into a $G$-subshift $Y_{\gamma}[X]$ which stores the information of $Y$ in a natural way. We shall show (\Cref{theorem_addition_formula}) that for any such embedding in which the cocycle induces free actions, the topological entropy satisfies the following addition formula,\[\htop(G \curvearrowright Y_{\gamma}[X]) = \htop(G \curvearrowright X)+\htop(H \curvearrowright Y). \]
	Furthermore, if both $X$ and $Y$ are SFTs, we have that $Y_{\gamma}[X]$ is an SFT. Therefore this formula can be used to embed the entropies of $H$-SFTs into the set of entropies of $G$-SFTs up to a fixed additive constant. We shall introduce the notion of group charts and give a proof of the addition formula on~\Cref{section_charts}.
	
	In~\Cref{section_reduce_ent_charts} we shall show that whenever a group chart is given by a $G$-SFT $X$, then we can choose it in such a way that its entropy is arbitrarily small~(\Cref{corollary_reducing_chart_entropy}). This will follow from a theorem that gives a canonical way of reducing the entropy of subshifts of finite type defined on an arbitrary countable amenable groups~(\Cref{theorem_tilings_forthewin}). We shall prove this theorem using the theory of quasitilings introduced by Ornstein and Weiss~\cite{OrnWei1987} and a recent result of Downarowicz, Huczec and Zhang~\cite{DownarowiczHuczekZhang2019}.
	
	In~\Cref{section_conditions_charts} we will characterize the existence of free charts, that is, charts for which every element of $x$ codes a true partition of $G$ into copies of $H$, through the notion of translation-like actions introduced by Whyte~\cite{Whyte1999}. Furthermore, following the ideas of Jeandel~\cite{Jeandel2015}, we shall show that whenever $H$ is finitely presented and there exists a non-empty $H$-SFT on which $H$ acts freely, then one can always find a free chart $(X,\gamma)$ for which $X$ is a $G$-SFT. Putting all of the previous results together, we shall show the following result.
	
	{
		\renewcommand{\thetheorem}{\ref{theorem_HG}}
		\begin{theorem}
			Let $G,H$ be finitely generated amenable groups and let $\entsft(H)$ and $\entsft(G)$ respectively denote the set of real numbers attainable as topological entropies of an SFT in each group. Suppose that 
			\begin{enumerate}
				\item $H$ admits a translation-like action on $G$.
				\item $H$ is finitely presented.
				\item There exists a non-empty $H$-SFT for which the $H$-action is free.
			\end{enumerate}
			Then, for every $\varepsilon >0$ there exists a $G$-SFT $X$ such that $h_{top}(G\curvearrowright X) < \varepsilon$ and \[h_{top}(G\curvearrowright X)+ \entsft(H) \subset \entsft(G).\]
		\end{theorem}
		\addtocounter{theorem}{-1}
	}	
	
	In~\Cref{section_characterization_Z2} we shall apply the above theorem to study the groups on which $\ZZ^2$ acts translation-like. It shall follow that modulo a computability obstruction, any finitely generated amenable group on which $\ZZ^2$ acts translation-like admits the same characterization of the set of numbers that can be attained as topological entropies of subshifts of finite type as $\ZZ^2$. Namely,

		{
			\renewcommand{\thetheorem}{\ref{theorem_caract_entropies_G_z2_translation_like}}
			\begin{theorem}
				Let $G$ be a finitely generated amenable group with decidable word problem which admits a translation-like action by $\ZZ^2$. The set of entropies attainable by $G$-subshifts of finite type is the set of non-negative upper semi-computable numbers.
			\end{theorem}
			\addtocounter{theorem}{-1}
		}
		
		Finally, in~\Cref{section_consequences} we shall use~\Cref{theorem_caract_entropies_G_z2_translation_like} to give a characterization of the numbers attainable as topological entropies of subshifts of finite type in several classes of groups. More precisely, we shall give a complete classification for polycyclic-by-finite groups~(\Cref{theorem_polycyclic}), products of two infinite and finitely generated amenable groups with decidable word problem~(\Cref{corollary_entropy_ofproducts}), countable amenable groups which admit a presentation with decidable word problem and a finitely generated subgroup on which $\ZZ^2$ acts translation-like~(\Cref{corollary_caract_entropies_full}) and infinite and finitely generated amenable branch groups with decidable word problem~(\Cref{theorem_branch_groups}).

\section{Preliminaries and notation}

In this note we shall consider left actions $G \curvearrowright X$ of countable amenable groups $G$ over compact metric spaces $X$ by homeomorphisms. Let us denote by $F\Subset G$ a finite subset of $G$ and by $1_G$ the identity of $G$. For $K \Subset G$ and $\varepsilon>0$ we say that $F \Subset G$ is left $(K,\varepsilon)$-invariant if $|KF \triangle F| \leq \varepsilon|F|$. From this point forward we shall omit the word left and plainly speak about $(K,\varepsilon)$-invariant sets. A sequence $\{F_n\}_{n \in \NN}$ of finite subsets of $G$ is called a F\o lner sequence if for every $K \Subset G$ and $\varepsilon>0$ the sequence is eventually $(K,\varepsilon)$-invariant.

\subsection{Shift spaces}

Let $\Sigma$ be a finite set and $G$ be a group. The set $\Sigma^G = \{ x\colon G \to \Sigma\}$ equipped with the left group action $G \curvearrowright X$ given by 
$gx(h) \isdef x(hg)$ is the \define{full $G$-shift}. The elements $a \in \Sigma$ and $x \in \Sigma^G$ are called \define{symbols} and \define{configurations} respectively. We endow $\Sigma^G$ with the product topology generated by the clopen subbase given by the \define{cylinders} $[a]_g \isdef \{x \in \Sigma^G~|~x(g) = a\}$. A \define{support} is a finite subset $F \Subset G$. Given a support $F$, a \define{pattern} with support $F$ is an element $p \in \Sigma^F$ and we write $\operatorname{supp}(p) = F$. We denote the cylinder generated by $p$ by $[p] = \bigcap_{h \in F}[p(h)]_{h}$.

A subset $X \subset \Sigma^G$ is a \define{$G$-subshift} if and only if it is $G$-invariant and closed in the product topology. Equivalently, $X$ is a $G$-subshift if and only if there exists a set of forbidden patterns $\FF$ such that\[X=X_\FF \isdef  {\Sigma^G \setminus \bigcup_{p \in \FF, g \in G} g[p]}.\]

Given a subshift $X \subset \Sigma^G$ and a support $F \Subset G$ the \define{language with support $F$} is the set $L_{F}(X) = \{ p \in \Sigma^F \mid [p] \cap X \neq \varnothing \}$ of all patterns which appear in some configuration $x \in X$. The \define{language} of $X$ is the set $L(X) = \bigcup_{F \Subset G}L_{F}(X)$.

\begin{remark}
	It is also possible to define the left $G$-action by $gx(h) \isdef x(g^{-1}h)$ instead of $x(hg)$. In this article we chose the latter in order to minimize the amount of superindices $^{-1}$ and to make the notation compatible with the setting of~\cite{DownarowiczHuczekZhang2019}, whose results we shall use to prove~\Cref{theorem_tilings_forthewin}. 
\end{remark}

\begin{definition}
	We say that a subshift $X$ is of \define{finite type (SFT)} if there exists a finite set $\FF$ of forbidden patterns such that $X = X_{\FF}$.
\end{definition}

\subsection{Topological entropy}

Let $G \curvearrowright X$ be the action of a group over a compact metrizable space by homeomorphisms. Given two open covers $\mathcal{U},\mathcal{V}$ of $X$ we define their \define{join} by $\mathcal{U} \vee \mathcal{V} = \{U \cap V \mid U \in \mathcal{U}, V\in \mathcal{V}   \}$. For $g \in G$ let $g\mathcal{U} = \{gU \mid U \in \mathcal{U}\}$ and denote by $N(\mathcal{U})$ the smallest cardinality of a subcover of $\mathcal{U}$. If $F$ is a finite subset of $G$, denote by $\mathcal{U}^F$ the join

$$\mathcal{U}^F = \bigvee_{g \in F}g^{-1}\mathcal{U}.$$ 

\begin{definition}
	Let $G \curvearrowright X$ be the action of a countable amenable group, $\mathcal{U}$ an open cover and $\{F_n\}_{n \in \NN}$ a F\o lner sequence for $G$. We define the \define{topological entropy of
	$G \curvearrowright X$ with respect to $\mathcal{U}$} as%
	\[
	h_{\text{top}}(G \curvearrowright X,\mathcal{U})=\lim_{n\rightarrow\infty}\frac{1}{\left\vert
		F_{n}\right\vert }\log N(\mathcal{U}^{F_{n}}).
	\]
	
\end{definition}

The function $F \mapsto \log N(\mathcal{U}^{F})$ is subadditive and thus the limit does not depend on the choice of F\o lner sequence, see for instance~\cite{OrnWei1987,Krieger2007_ornsteinweiss}. The \define{topological entropy} of $G \curvearrowright X$ is defined as \[
h_{\text{top}}(G \curvearrowright X)=\sup_{\mathcal{U}}h_{\text{top}}(G \curvearrowright X,\mathcal{U}).
\]

In the case where $G \curvearrowright X$ is expansive, any open cover $\mathcal{U}$ whose elements have diameter less than the expansivity constant achieves the supremum. Particularly, in the case of a subshift $X \subset \Sigma^G$ we may consider the partition $\xi = \{ [a]_{1_G} \mid a \in \Sigma \}$. For a finite $F \subset G$ we obtain that $\xi^F = \{ [p] \mid p \in L_F(X) \}$. Hence, whenever $X$ is a subshift its topological entropy can be computed by \[
h_{\text{top}}(G \curvearrowright X)=\lim_{n\rightarrow\infty}\frac{1}{\left\vert
	F_{n}\right\vert }\log(|L_{F_n}(X)|).
\]

A more intuitive way to understand this limit, is that the function $F \mapsto \frac{1}{\left\vert
	F\right\vert }\log(|L_{F}(X)|)$ converges as $F$ becomes more and more invariant, that is, for every $\varepsilon>0$ there exists $K \Subset G$ and $\delta>0$ such that for any $(K,\delta)$-invariant set $F$ we have $|h_{\text{top}}(G \curvearrowright X) - \frac{1}{\left\vert
	F\right\vert }\log(|L_{F}(X)|) | \leq \varepsilon$. For a self contained proof and relevant background see~\cite[Theorem 4.38]{KerrLiBook2016}.

In the case when the open cover $\mathcal{U}$ consists of pairwise disjoint open sets, it can be shown that the function $F \mapsto \log N(\mathcal{U}^{F})$ is not only subadditive, but satisfies Shearer's inequality (see~\cite[Corollary 6.2]{DownFrejRomag2015}). This in turn implies that in the case of a subshift we may write: \begin{align}\label{eq_entropyforidiots}
h_{\text{top}}(G \curvearrowright X)=\inf_{F \in \mathcal{F}(G)}\frac{1}{\left\vert
	F\right\vert }\log(|L_{F}(X)|).
\end{align}

where $\mathcal{F}(G)$ denotes the set of all finite subsets of $G$, see~\cite[Corollary 6.3]{DownFrejRomag2015}). 

\begin{remark}
	In fact the result that topological entropy can be computed as an infimum over all finite subsets holds for any $G\curvearrowright X$, although it may not hold individually for every partition $\mathcal{U}$. This was proven in~\cite{DownFrejRomag2015} using the variational principle. A good way to think about it is that in the context of amenable groups, the topological entropy coincides with the naive entropy of Burton~\cite{burton2017naive}.
\end{remark}

Let us introduce the following notation which will be useful in the remainder of the article. For a group $G$, we denote the set of real numbers attained as topological entropies of $G$-SFTs by $\entsft(G)$.

\[ \entsft(G) = \{r \in \RR \mid \textrm{there exists a } G\textrm{-SFT } X, \htop(G \curvearrowright X) = r \}   \]

 Let us state two classical theorems from the literature which will be used further on. Recall that a \define{Perron number} is a real algebraic integer greater than $1$ and greater than the modulus of its algebraic conjugates. 

\begin{theorem}[Lind~\cite{Lind1984}]
	$\entsft(\ZZ)$ is the set of non-negative rational multiples of logarithms of Perron numbers.
\end{theorem}

In order to state the second result, we need to introduce the notion of upper semi-computable numbers, they are also sometimes called ``right-recursively enumerable numbers''.

\begin{definition}
	A real number $r$ is \define{upper semi-computable} if there exists a Turing machine $T$ which on input $n \in \NN$ halts with the coding of a rational number $q_n \geq r$ on its tape such that $\lim_{n \to \infty} q_n = r$.
\end{definition}

\begin{theorem}[Hochman and Meyerovitch~\cite{HochmanMeyerovitch2010}]\label{theorem_HochmanTom}
	For $d \geq 2$, $\entsft(\ZZ^d)$ is the set of non-negative upper semi-computable numbers.
\end{theorem}

\section{Realization of entropies of subshifts of finite type}

\subsection{Group charts and the addition formula}\label{section_charts}

\begin{definition}
	Let $G,H$ be two topological groups and let $X$ be a compact topological space on which $G$ acts on the left by homeomorphisms. A continuous map $\gamma \colon H \times X \to G$ is called an \define{$H$-cocycle} if it satisfies the equation \[  \gamma(h_1h_2,x) =  \gamma(h_1,\gamma(h_2,x)x)\cdot \gamma(h_2,x) \mbox{ for every $h_1,h_2$ in $H$}.\]
\end{definition}

The cocycle equation can be represented by the diagram shown on~\Cref{fig:diagram_cocycle}. Let us clarify how this equation fits within the classical setting of cocycles. A continuous map $\gamma$ as above induces an action $H \curvearrowright X$ by setting $h \cdot x = \gamma(h,x)x$, where the product on the right is the one associated to the action $G \curvearrowright X$. With this action $H \curvearrowright X$ in mind, the equation simplifies to the better known equation for cocycles \[ \gamma(h_1h_2,x) =  \gamma(h_1,h_2 \cdot x)\cdot \gamma(h_2,x) \mbox{ for every $h_1,h_2$ in $H$}.   \]

Any $H$-cocycle $\gamma$ induces a family $\{H \overset{x}{\curvearrowright} G\}_{x \in X}$ of left $H$-actions on $G$. Indeed, if for fixed $x \in X$ we define for $h \in H$ and $g \in G$, the action given by $h \cdot_x g \isdef \gamma(h,gx)g$, then for all $h_1,h_2 \in H$ we have \begin{align*}
(h_1h_2) \cdot_x g & = \gamma(h_1h_2,gx)g\\ & = (\gamma(h_1,\gamma(h_2,gx)gx)\cdot \gamma(h_2,gx) )g \\ & =\gamma(h_1,(\gamma(h_2,gx)g)x)\cdot (\gamma(h_2,gx)g) \\ & = h_1 \cdot_x (\gamma(h_2,gx)g) \\ & = h_1 \cdot_x (h_2 \cdot_x g).
\end{align*}

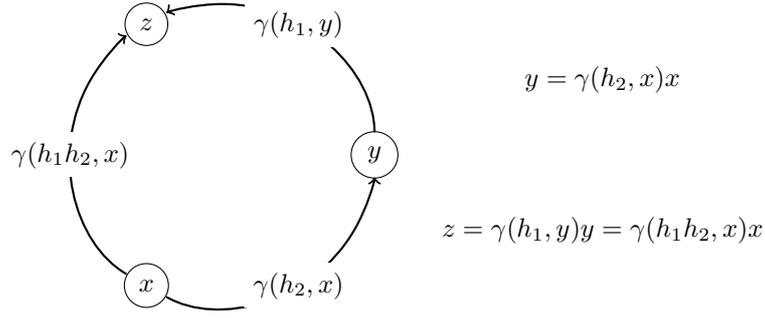
\begin{figure}[h!]
	\centering
	\begin{tikzpicture}
	\node[circle, draw] (A) at (240: 2cm)  {$x$};
	\node[circle, draw] (B) at (0:2cm)  {$y$};
	\node[circle, draw] (C) at (120:2cm) {$z$};
	\draw [->, thick, shorten >=0.3cm,shorten <=0.3cm] (A) arc (240:360:2cm) node[midway,fill=white] {$\gamma(h_2,x)$};
	\draw [->, thick, shorten >=0.3cm,shorten <=0.3cm] (A) arc (240:120:2cm) node[midway,fill=white] {$\gamma(h_1h_2,x)$};
	\draw [->, thick, shorten >=0.3cm,shorten <=0.3cm] (B) arc (0:120:2cm) node[midway, fill=white] {$\gamma(h_1,y)$};
	
	\node at (5,1) {$y = \gamma(h_2,x)x$};
	\node at (5,-1) {$z = \gamma(h_1,y)y = \gamma(h_1h_2,x)x$};
	\end{tikzpicture}
	\caption{The circles $x,y,z$ represent points in the space $X$ while the arrows represent left multiplication by group elements. The cocycle equation states that the arrows commute: $\gamma(h_1h_2,x) = \gamma(h_1,y)\gamma(h_2,x)$.}
	\label{fig:diagram_cocycle}
\end{figure}

\begin{remark}
	If $H$ is a finitely generated group and $S$ a finite generating set for $H$, then the values of any $H$-cocycle $\gamma$ restricted to $S \times X$ define $\gamma$ completely. Furthermore, whenever $G$ is countable, by continuity of $\gamma$ and compactness of $S \times X$ we have that $\gamma$ must be uniformly bounded on $S \times X$ and thus $\gamma(S \times X) \Subset G$. Hence if $X$ is a $G$-subshift, there exists a finite set $F\Subset G$ such that $\gamma$ restricted to $S \times X$ is completely defined by a finite map $\tilde{\gamma} \colon S \times L_{F}(X)$.
\end{remark}

The following notion is strongly motivated by the work of Jeandel~\cite{Jeandel2015}.

\begin{definition}
	Let $G,H$ be two countable groups. Given a left action $G\curvearrowright X$ and an $H$-cocycle $\gamma \colon H \times X \to G$ we say the pair $(X,\gamma)$ is a \define{$G$-chart} of $H$. Furthermore, if for each $x \in X$ the action $H \overset{x}{\curvearrowright} G$ is free, we say that $(X,\gamma)$ is a \define{free $G$-chart} of $H$.
\end{definition}

\begin{example}\label{example_obviouschart}
	The trivial system $G \curvearrowright \{0\}$ consisting of a single point and the cocycle $\gamma\colon H \times \{0\} \to G$ which sends $(h,0) \mapsto h$ is a free $G$-chart of $H$ for any subgroup $H \leq G$.\hfill\exampleqed
\end{example}

\begin{example}\label{example_snake}
	Let $G = \ZZ^2$ and let $\Sigma_{\texttt{snake}}$ be the set of vector pairs given by
	\[\Sigma_{\texttt{snake}} = \{ (\ell,r) \in \{(1,0),(-1,0),(0,1),(0,-1)\}^2  \mid \ell \neq r   \} \]
	Visually, we may represent $\Sigma_{\texttt{snake}}$ by the set of square unit tiles shown on~\Cref{fig:tiles_Z}. The first vector is represented by the tail of the arrow and the second vector by the outgoing arrow.
	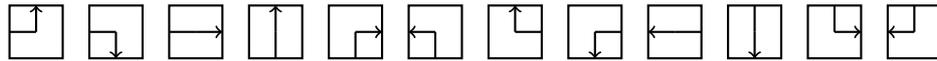
\begin{figure}[h!]
		\centering
		\begin{tikzpicture}[scale=0.7]

\def \outN {
	\draw[thick, ->] (0.5,0.5) -- (0.5,1);
}
\def \outS {
	\draw[thick, ->] (0.5,0.5) -- (0.5,0);
}
\def \outW {
	\draw[thick, ->] (0.5,0.5) -- (0,0.5);
}
\def \outE {
	\draw[thick, ->] (0.5,0.5) -- (1,0.5);
}
\def \inN {
	\draw[thick] (0.5,1) -- (0.5,0.5);
}
\def \inS {
	\draw[thick] (0.5,0) -- (0.5,0.5);
}
\def \inW {
	\draw[thick] (0,0.5) -- (0.5,0.5);
}
\def \inE {
	\draw[thick] (1,0.5) -- (0.5,0.5);
}
\def \paintgreen{
	\filldraw[green!20!white] (0,0) rectangle (1,1);
}
\def \paintred{
	\filldraw[red!20!white] (0,0) rectangle (1,1);
}
\def \cuadri{
	\draw [thick] (0,0) rectangle (1,1);
}

\begin{scope}[rotate = 0]

\begin{scope}[shift = {(0,0)}]
\cuadri \inW \outN 
\end{scope}

\begin{scope}[shift = {(1.5,0)}]
\cuadri \inW \outS 
\end{scope}

\begin{scope}[shift = {(3,0)}]
\cuadri \inW \outE
\end{scope}

\begin{scope}[shift = {(4.5,0)}]
\cuadri \inS \outN 
\end{scope}

\begin{scope}[shift = {(6,0)}]
\cuadri \inS \outE 
\end{scope}

\begin{scope}[shift = {(7.5,0)}]
\cuadri \inS \outW
\end{scope}

\end{scope}

\begin{scope}[rotate = 0, shift = {(9,0)}]

\begin{scope}[shift = {(0,0)}]
\cuadri \inE \outN 
\end{scope}

\begin{scope}[shift = {(1.5,0)}]
\cuadri \inE \outS 
\end{scope}

\begin{scope}[shift = {(3,0)}]
\cuadri \inE \outW
\end{scope}

\begin{scope}[shift = {(4.5,0)}]
\cuadri \inN \outS 
\end{scope}

\begin{scope}[shift = {(6,0)}]
\cuadri \inN \outE 
\end{scope}

\begin{scope}[shift = {(7.5,0)}]
\cuadri \inN \outW
\end{scope}

\end{scope}
\end{tikzpicture}
		\caption{The alphabet $\Sigma_{\texttt{snake}}$.}
		\label{fig:tiles_Z}
	\end{figure}
	
	For $a=(\ell,r) \in \Sigma_{\texttt{snake}}$ let $L(a)=\ell$ and $R(a)=r$. We define the \define{snake shift} as the $\ZZ^2$-SFT  $X_{\texttt{snake}} \subset (\Sigma_{\texttt{snake}})^{\ZZ^2}$ of all configurations $x$ such that for every position $v \in \ZZ^2$, we have $R(x(v)) = L(x(v + R(x(v))))$ and $L(x(v)) = R(x(v + L(x(v))))$. Visually, these are the configurations such that every outgoing arrow matches with an incoming arrow. Let $\gamma_{\texttt{snake}}\colon \ZZ \times X \to \ZZ^2$ be the $\ZZ$-cocycle defined by $\gamma_{\texttt{snake}}(1,x) = R(x((0,0)))$ and $\gamma_{\texttt{snake}}(-1,x) = L(x((0,0)))$. It can be verified that $(X_{\texttt{snake}},\gamma_{\texttt{snake}})$ is a $\ZZ^2$-chart of $\ZZ$.
	
	The $\ZZ^2$-chart $(X_{\texttt{snake}},\gamma_{\texttt{snake}})$ of $\ZZ$ is not free. Indeed, every configuration $x$ in which a cycle appears induces an action $\ZZ \overset{x}{\curvearrowright} \ZZ^2$ which is not free. Let $X^{\texttt{free}}_{\texttt{snake}} \subset X_{\texttt{snake}}$ be the \define{free snake} subshift consisting of all configurations $x \in X_{\texttt{snake}}$ such that no cycles appear, it can be verified that $(X^{\texttt{free}}_{\texttt{snake}}, \gamma_{\texttt{snake}}|_{\ZZ \times X^{\texttt{free}}_{\texttt{snake}}})$ is a free $\ZZ^2$-chart of $\ZZ$. See~\Cref{fig:tiles_Z_example}.\hfill\exampleqed\end{example}

\begin{figure}[h!]
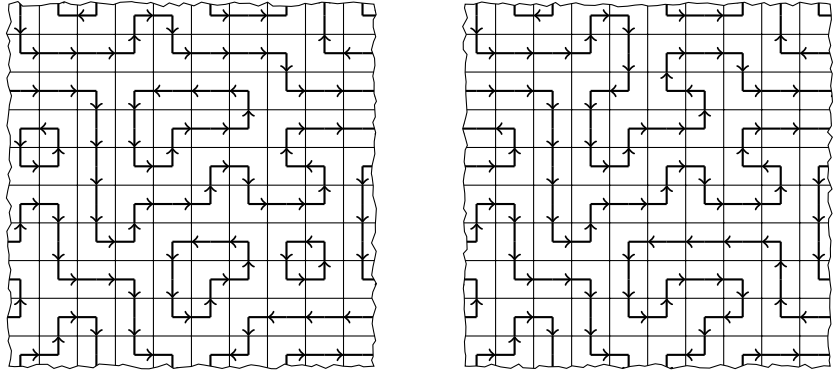

	\centering
	\include{tiles_Z_example}
	\caption{On the left we see a local patch of $X_{\texttt{snake}}$. The value of the $\ZZ$-cocycle $\gamma_{\texttt{snake}}(n,x)$ corresponds to the vector of $\ZZ^2$ obtained by following the arrow at the origin $n$ times. On the right we see a local patch of a configuration of $X^{\texttt{free}}_{\texttt{snake}}$. As cycles are forbidden, the cocycle induces a free action.}
	\label{fig:tiles_Z_example}
\end{figure}

Let $\Sigma$ be a set. The notion of $G$-chart gives canonical way to recover an $H$-orbit of $\Sigma$ given a $G$-orbit $y \in \Sigma^G$ and basepoints $x \in X$ and $g \in G$. Indeed, if $(X,\gamma)$ is a $G$-chart of $H$ we can associate to every $y \in \Sigma^G$ an orbit  $\pi_{x,g}(y) \in \Sigma^H$ by setting \[  \pi_{x,g}(y)(h) \isdef y(h \cdot_x g) = y(\gamma(h,gx)g) \mbox{ for every $h$ in $H$.}\]

Moreover, this configuration satisfies that for every $h_1,h_2 \in H$:
\begin{align*}
(h_2\pi_{x,g}(y))(h_1) & = \pi_{x,g}(y)(h_1h_2) \\
& = y((h_1h_2) \cdot_x g)\\
& = y(h_1\cdot_x (h_2 \cdot_x g))\\
& = (\pi_{x,h_2 \cdot_x g}(y))(h_1)
\end{align*}In other words, the left shift action of $h_2$ on $\pi_{x,g}(y)$ is the same as $\pi_{x,h_2 \cdot_x g}(y)$, that is, the configuration obtained by changing the basepoint $g$ by $h_2 \cdot_x g$.

From now on, we shall only consider $G$-charts $(X,\gamma)$ where $X$ is a $G$-subshift.

\begin{definition}
	Let $(X,\gamma)$ be a $G$-chart of $H$ and $Y\subset \Sigma^H$ be an $H$-subshift. The \define{$(X,\gamma)$-embedding of $Y$} is the $G$-subshift $Y_{\gamma}[X] \subset \Sigma^G \times X$ which has the property that $(y,x) \in Y_{\gamma}[X]$ if and only if for every $g \in G$ then $\pi_{x,g}(y)$ is in $Y$.
\end{definition}

In simpler words, $Y_{\gamma}[X]$ is the subshift of all pairs $(y,x)$ where $x \in X$ and every copy of $H$ induced by the action $H \overset{x}{\curvearrowright} G$ is decorated independently with a configuration from $Y$.

\newcommand{\alfTriang}{\begin{tikzpicture}[scale = 0.5]
	\draw[fill = black!5] (0.2,0.2) -- (0.8,0.2) -- (0.5,0.8) -- cycle;
	\end{tikzpicture}}
\newcommand{\alfCircle}{\begin{tikzpicture}[scale = 0.5]
	\draw[fill = black!5] (0.5,0.5) circle (0.3);
	\end{tikzpicture}}
\newcommand{\alfSquare}{\begin{tikzpicture}[scale = 0.5]
	\draw[fill = black!5] (0.2,0.2) rectangle (0.8,0.8);
	\end{tikzpicture}}

\begin{example}
	Consider the free $\ZZ^2$-chart $(X,\gamma)$ of $\ZZ$ from~\Cref{example_snake}, that is, $X = X^{\texttt{free}}_{\texttt{snake}}$ and $\gamma = \gamma_{\texttt{snake}}|_{\ZZ \times X^{\texttt{free}}_{\texttt{snake}}}$. Consider the $\ZZ$-subshift $Y$ consisting on the orbit of the sequence $x$ over the alphabet $\Sigma = \{\alfSquare, \alfTriang, \alfCircle \}$ given by \[
	x(n) = \begin{cases}
	\alfSquare & \mbox{ if } n =0 \bmod{3}\\
	\alfTriang & \mbox{ if } n =1 \bmod{3}\\
	\alfCircle & \mbox{ if } n =2 \bmod{3}.
	\end{cases}
	\]
	
	\begin{figure}[h!]
		\centering
		\begin{tikzpicture}[scale=0.6]
		
		\def \outN {
			\draw[thick, ->] (0.5,0.5) -- (0.5,1);
		}
		\def \outS {
			\draw[thick, ->] (0.5,0.5) -- (0.5,0);
		}
		\def \outW {
			\draw[thick, ->] (0.5,0.5) -- (0,0.5);
		}
		\def \outE {
			\draw[thick, ->] (0.5,0.5) -- (1,0.5);
		}
		\def \inN {
			\draw[thick] (0.5,1) -- (0.5,0.5);
		}
		\def \inS {
			\draw[thick] (0.5,0) -- (0.5,0.5);
		}
		\def \inW {
			\draw[thick] (0,0.5) -- (0.5,0.5);
		}
		\def \inE {
			\draw[thick] (1,0.5) -- (0.5,0.5);
		}
		\def \pG {
			\draw[fill = black!5] (0.2,0.2) rectangle (0.8,0.8);
		}
		\def \pR {
			\draw[fill = black!5] (0.5,0.5) circle (0.3);
		}
		\def \pB {
			\draw[fill = black!5] (0.2,0.2) -- (0.8,0.2) -- (0.5,0.8) -- cycle;
		}
		
		\clip[draw,decorate,decoration={random steps, segment length=3pt, amplitude=1pt}] (0.2,0.2) rectangle (9.8,9.8); 
		\begin{scope}[shift = {(9,9)}] \pG \inN \outE \end{scope}
		\begin{scope}[shift = {(0,0)}] \pR \inS \outE \end{scope}
		\begin{scope}[shift = {(1,0)}] \pG \inW \outN \end{scope}
		\begin{scope}[shift = {(1,1)}] \pB \inS \outE \end{scope}
		\begin{scope}[shift = {(2,1)}] \pR \inW \outS \end{scope}
		\begin{scope}[shift = {(2,0)}] \pG \inN \outS \end{scope}
		\begin{scope}[shift = {(0,1)}] \pG \inW \outN \end{scope}
		\begin{scope}[shift = {(0,2)}] \pB \inS \outW \end{scope}
		\begin{scope}[shift = {(0,3)}] \pB  \inW \outN \end{scope}
		\begin{scope}[shift = {(0,4)}] \pR \inS \outE \end{scope}
		\begin{scope}[shift = {(1,4)}] \pG \inW \outS \end{scope}
		\begin{scope}[shift = {(1,3)}] \pB  \inN \outS \end{scope}
		\begin{scope}[shift = {(1,2)}] \pR \inN \outE \end{scope}
		\begin{scope}[shift = {(2,2)}] \pG \inW \outE \end{scope}
		\begin{scope}[shift = {(3,2)}] \pB \inW \outS \end{scope}
		\begin{scope}[shift = {(3,1)}] \pR \inN \outS \end{scope}
		\begin{scope}[shift = {(3,0)}] \pG \inN \outE \end{scope}
		\begin{scope}[shift = {(4,0)}] \pB \inW \outS \end{scope}
		\begin{scope}[shift = {(0,5)}] \pR \inW \outE \end{scope}
		\begin{scope}[shift = {(1,5)}] \pG \inW \outN \end{scope}
		\begin{scope}[shift = {(1,6)}] \pB \inS \outW \end{scope}
		\begin{scope}[shift = {(0,6)}] \pR \inE \outW \end{scope}
		\begin{scope}[shift = {(0,7)}] \inW \outE \pR \end{scope}
		\begin{scope}[shift = {(1,7)}] \pG \inW \outE \end{scope}
		\begin{scope}[shift = {(2,7)}] \pB \inW \outS \end{scope}
		\begin{scope}[shift = {(2,6)}] \pR \inN \outS \end{scope}
		\begin{scope}[shift = {(2,5)}] \pG \inN \outS \end{scope}
		\begin{scope}[shift = {(2,4)}] \pB \inN \outS \end{scope}
		\begin{scope}[shift = {(2,3)}] \pR \inN \outE \end{scope}
		\begin{scope}[shift = {(3,3)}] \pG \inW \outN \end{scope}
		\begin{scope}[shift = {(3,4)}] \pB \inS \outE \end{scope}
		\begin{scope}[shift = {(4,4)}] \pR \inW \outE \end{scope}
		\begin{scope}[shift = {(5,4)}] \pG \inW \outN \end{scope}
		\begin{scope}[shift = {(5,5)}] \pB \inS \outE \end{scope}
		\begin{scope}[shift = {(6,5)}] \pR \inW \outS \end{scope}
		\begin{scope}[shift = {(6,4)}] \pG \inN \outE \end{scope}
		\begin{scope}[shift = {(7,4)}] \pB \inW \outE \end{scope}
		\begin{scope}[shift = {(8,4)}] \pR \inW \outN \end{scope}
		\begin{scope}[shift = {(8,5)}] \pG \inS \outW \end{scope}
		\begin{scope}[shift = {(7,5)}] \pB \inE \outN \end{scope}
		\begin{scope}[shift = {(7,6)}] \pR \inS \outE \end{scope}
		\begin{scope}[shift = {(8,6)}] \pG \inW \outE \end{scope}
		\begin{scope}[shift = {(9,6)}] \pB \inW \outE \end{scope}
		\begin{scope}[shift = {(9,5)}] \pB \inE \outS \end{scope}
		\begin{scope}[shift = {(9,4)}] \pR \inN \outS \end{scope}
		\begin{scope}[shift = {(9,3)}] \pG \inN \outS \end{scope}
		\begin{scope}[shift = {(9,2)}] \pB \inN \outE \end{scope}
		\begin{scope}[shift = {(9,1)}] \pB \inE \outW \end{scope}
		\begin{scope}[shift = {(8,1)}] \pR \inE \outN \end{scope}
		\begin{scope}[shift = {(8,2)}] \pG \inS \outN \end{scope}
		\begin{scope}[shift = {(8,3)}] \pB \inS \outW \end{scope}
		\begin{scope}[shift = {(7,3)}] \pR \inE \outW \end{scope}
		\begin{scope}[shift = {(6,3)}] \pG \inE \outW \end{scope}
		\begin{scope}[shift = {(5,3)}] \pB \inE \outW \end{scope}
		\begin{scope}[shift = {(4,3)}] \pR \inE \outS \end{scope}
		\begin{scope}[shift = {(4,2)}] \pG \inN \outS \end{scope}
		\begin{scope}[shift = {(4,1)}] \pB \inN \outE \end{scope}
		\begin{scope}[shift = {(5,1)}] \pR \inW \outN \end{scope}
		\begin{scope}[shift = {(5,2)}] \pG \inS \outE \end{scope}
		\begin{scope}[shift = {(6,2)}] \pB \inW \outE \end{scope}
		\begin{scope}[shift = {(7,2)}] \pR \inW \outS \end{scope}
		\begin{scope}[shift = {(7,1)}] \pG \inN \outW \end{scope}
		\begin{scope}[shift = {(6,1)}] \pB \inE \outS \end{scope}
		\begin{scope}[shift = {(6,0)}] \pR \inN \outW \end{scope}
		\begin{scope}[shift = {(5,0)}] \pG \inE \outS \end{scope}
		\begin{scope}[shift = {(7,0)}] \pG \inS \outE \end{scope}
		\begin{scope}[shift = {(8,0)}] \pB \inW \outE \end{scope}
		\begin{scope}[shift = {(9,0)}] \pR \inW \outE \end{scope}
		\begin{scope}[shift = {(0,9)}] \pG \inN \outS \end{scope}
		\begin{scope}[shift = {(0,8)}] \pB \inN \outE \end{scope}
		\begin{scope}[shift = {(1,8)}] \pR \inW \outE \end{scope}
		\begin{scope}[shift = {(2,8)}] \pG \inW \outE \end{scope}
		\begin{scope}[shift = {(3,8)}] \pB \inW \outN \end{scope}
		\begin{scope}[shift = {(3,9)}] \pR \inS \outE \end{scope}
		\begin{scope}[shift = {(4,9)}] \pG \inW \outS \end{scope}
		\begin{scope}[shift = {(4,8)}] \pB \inN \outS \end{scope}
		\begin{scope}[shift = {(4,7)}] \pR \inN \outW \end{scope}
		\begin{scope}[shift = {(3,7)}] \pG \inE \outS \end{scope}
		\begin{scope}[shift = {(3,6)}] \pB \inN \outS \end{scope}
		\begin{scope}[shift = {(3,5)}] \pR \inN \outE \end{scope}
		\begin{scope}[shift = {(4,5)}] \pG \inW \outN \end{scope}
		\begin{scope}[shift = {(4,6)}] \pB \inS \outE \end{scope}
		\begin{scope}[shift = {(5,6)}] \pR \inW \outE \end{scope}
		\begin{scope}[shift = {(6,6)}] \pG \inW \outN \end{scope}
		\begin{scope}[shift = {(6,7)}] \pB \inS \outW \end{scope}
		\begin{scope}[shift = {(5,7)}] \pR \inE \outN \end{scope}
		\begin{scope}[shift = {(5,8)}] \pG \inS \outE \end{scope}
		\begin{scope}[shift = {(6,8)}] \pB \inW \outE \end{scope}
		\begin{scope}[shift = {(7,8)}] \pR \inW \outS \end{scope}
		\begin{scope}[shift = {(7,7)}] \pG \inN \outE \end{scope}
		\begin{scope}[shift = {(8,7)}] \pB \inW \outE \end{scope}
		\begin{scope}[shift = {(9,7)}] \pR \inW \outE \end{scope}
		\begin{scope}[shift = {(2,9)}] \pR \inN \outW \end{scope}
		\begin{scope}[shift = {(1,9)}] \pG \inE \outN \end{scope}
		\begin{scope}[shift = {(9,8)}] \pR \inE \outW \end{scope}
		\begin{scope}[shift = {(8,8)}] \pG \inE \outN \end{scope}
		\begin{scope}[shift = {(8,9)}] \pB \inS \outN \end{scope}
		\begin{scope}[shift = {(5,9)}] \pR \inN \outE \end{scope}
		\begin{scope}[shift = {(6,9)}] \pG \inW \outE \end{scope}
		\begin{scope}[shift = {(7,9)}] \pB \inW \outN \end{scope}
		\draw (0,0) grid (10,10);
		\end{tikzpicture}
		\caption{The subshift $Y_{\gamma}[X]$ is obtained by ``overlaying'' the copies of $H$ induced by $\gamma$ on $X$ with configurations of $Y$.}
		\label{fig:tiles_Z_examplecolor}
	\end{figure}
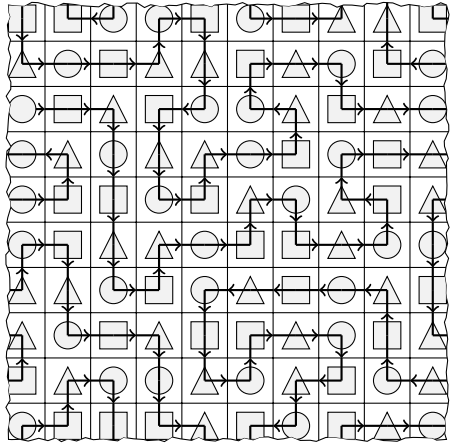
	
	The subshift $Y_{\gamma}[X]$ is the set of all configurations $(y,x) \in \{\alfSquare, \alfTriang, \alfCircle \}^{\ZZ^2} \times X$ such that every path in $x \in X$ induced by $\gamma$ is decorated independently with a configuration from $Y$, see~\Cref{fig:tiles_Z_examplecolor}. \hfill\exampleqed
\end{example}

\begin{remark}\label{remark_SFTchart_SFT_is_SFT}
	Let $(X,\gamma)$ be a $G$-chart of $H$. If $X$ is a $G$-SFT and $Y$ is an $H$-SFT then $Y_{\gamma}[X]$ is also a $G$-SFT.
\end{remark}

\begin{remark}\label{remark_chart_having_one_free_action_is_good}
	If $(X,\gamma)$ is a $G$-chart of $H$ and there is $x \in X$ such that $H \overset{x}{\curvearrowright} G$ is free, then the map $\pi\colon Y_{\gamma}[X] \to Y$ given by $\pi(y,x) = \pi_{x,1_G}(y)$ is surjective. In particular, $Y_{\gamma}[X]$ is non-empty if and only if $Y$ is non-empty.
\end{remark}

The following result is the main tool that will allow us to take a subshift of finite type with fixed topological entropy defined on a group $H$, and realize it, modulo a fixed constant, as the topological entropy of a subshift of finite type defined on any group where $H$ can be freely charted . It shows that the entropy of any subshift which is embedded in a free chart can be expressed through an addition formula.

\begin{theorem}[addition formula]\label{theorem_addition_formula}
	Let $G,H$ be countable amenable groups. For any free $G$-chart $(X,\gamma)$ of $H$ and for any $H$-subshift $Y$ we have
	\begin{equation} \htop(G \curvearrowright Y_{\gamma}[X]) = \htop(H \curvearrowright Y)+ \htop(G \curvearrowright X).\end{equation}
\end{theorem}

\begin{proof}
	Denote by $\Sigma_X$ and $\Sigma_Y$ the alphabets of $X$ and $Y$ respectively. Let $\varepsilon >0$. There exists $S \Subset H$ and $\delta>0$ such that every non-empty $(S,\delta)$-invariant set $F \Subset H$ satisfies \begin{equation*}
	\ee^{\htop(H \curvearrowright Y)|F|} \leq |L_F(Y)| \leq \ee^{(\htop(H \curvearrowright Y)+\varepsilon)|F|}.
	\end{equation*}
	Let $\gamma\colon H \times X \to G$ be the $H$-cocycle associated to $X$. As $G$ is countable and $S$ is finite, the restriction of $\gamma$ to $S \times X$ is bounded. Let $W_1 \Subset G$ be a set such that $\gamma(S \times X) \subset W_1$. By continuity of $\gamma$ there exists $W_2 \Subset G$ such that a set such that for every $s \in S$ we have $\gamma(s,x) = \gamma(s,y)$ whenever $x|_{W_2} = y|_{W_2}$. For every $\varepsilon'>0$ there exists a finite set $S' \supset W_1 \cup W_2$ and $\delta'>0$ such that every non-empty $(S',\delta')$-invariant  $F' \Subset G$ satisfies \begin{equation*}
	\ee^{\htop(G \curvearrowright X)|F'|} \leq |L_{F'}(X)| \leq \ee^{(\htop(G \curvearrowright X)+\varepsilon')|F'|}.
	\end{equation*}
	
	Let $F' \Subset G$ be an $(S',\delta')$-invariant set and consider a pattern $p \in L_{S'F'}(X)$. As $W_2 \subset S'$, for each $f' \in F'$ and $s \in S$ the map $\gamma_{p}(s,f') \isdef \gamma(s,f'x)$ where $x\in X$ is any configuration such that $x|_{S'F'} = p$ is well defined. Let us define the relation $R \subset F' \times F'$ as the smallest equivalence relation such that whenever $f_1',f_2' \in F'$ satisfy that for some $s_1,s_2 \in S$ we have $\gamma_{p}(s_1,f_1')f_1'=\gamma_{p}(s_2,f_2')f_2'$, then $(f_1',f_2')\in R$.

	 The equivalence relation $R$ induces a partition $F' = F^{p}_1 \uplus F^{p}_2 \uplus \dots \uplus F^{p}_{k(p)}$. Let us denote by $\partial_S F^{p}_i$ the set of all $g' \in S'F' \setminus F'$ for which there is $f' \in F^{p}_i$ and $s \in S$ such that $\gamma_{p}(s,f')f' = g'$. By definition of $R$, note that the sets $\partial_S F^{p}_i$ are pairwise disjoint and $\partial_S F^{p}_i \subset S'F' \setminus F'$.
	
	We obtain that $\sum_{i = 1}^{k(p)} |\partial_S F^{p}_i| \leq |S'F' \setminus F'|\leq \delta' |F'|$. Dividing both sides by $|F'|$ and multiplying each left term by $\frac{|F^{p}_i|}{|F^{p}_i|}$ we obtain: \begin{equation*}
	\sum_{i = 1}^{k(p)} \frac{|\partial_S F^{p}_i|}{|F^{p}_i|}\frac{|F^{p}_i|}{|F'|} \leq \delta'.
	\end{equation*}

	  Denote by $\mu_i$ the ratio $\mu_i = \frac{|F^{p}_i|}{|F'|}$ and by $\delta_i = \frac{|\partial_S F^{p}_i|}{|F^{p}_i|}$. Note that $\mu_i \in [0,1]$, $\sum_{i = 1}^{k(p)}\mu_i = 1$ and $\delta_i \in [0,|S|]$. Let $I(p)$ be the set of indices such that $\delta_i \leq \delta$. We have that $\sum_{i \in I(p)}\delta_i \mu_i + \sum_{j \notin I(p)}\delta_j \mu_j \leq \delta'$. A simple manipulation of this expression yields \begin{equation}\label{equation_good_proportion}
	  \sum_{i \in I(p)}\mu_i \geq 1-\frac{\delta'}{\delta}.
	  \end{equation}
	  
	  The intuitive meaning of~\Cref{equation_good_proportion} is that the total amount of sites in the $(S',\varepsilon')$-invariant set $F'$ which lie in an induced subset of $H$ which is $(S,\delta)$-invariant can be made arbitrarily large by tweaking the ratio $\frac{\delta'}{\delta}$. 

	  As the $G$-chart $(X,\gamma)$ of $H$ is free, we can identify each set $F_i^{p}$ with a subset $H^p_i \Subset H$ and $\partial_S F^{p}_i$ with $SH^p_i \setminus H^p_i$. Furthermore, we have $|H^p_i| = |F_i^{p}|$ and $|SH^p_i \setminus H^p_i| = |\partial_S F^p_i|$. In other words, whenever $|\partial_S F^p_i| \leq \delta |F^p_i|$, the set $H_i$ is $(S,\delta)$-invariant. Now we use this computation to estimate the size of $|L_{F'}(Y_{\gamma}[X])|$. Clearly $|L_{S'F'}(X)| \geq |L_{F'}(X)|$, we can thus obtain \begin{align*}
	  |L_{F'}(Y_{\gamma}[X])| &  \leq \sum_{p\in L_{S'F'}(X)}\prod_{i=1 }^{k(p)}|L_{H^p_i}(Y)| \\
	  & \leq \sum_{p\in L_{S'F'}(X)}\prod_{i\in I(p)}|L_{H^p_i}(Y)|\prod_{j\notin I(p)}|L_{H^p_j}(Y)|\\
	  & \leq \sum_{p\in L_{S'F'}(X)}\prod_{i\in I(p)}|L_{H^p_i}(Y)|\prod_{j\notin I(p)}|\Sigma_Y|^{|H^p_j|}\\
	  & \leq |\Sigma_Y|^{\frac{\delta'|F'|}{\delta}} \sum_{p\in L_{S'F'}(X)}\prod_{i\in I(p)}|L_{H^p_i}(Y)|.
	  \end{align*}
	  As each $H^p_i$ for $i \in I(p)$ is $(S,\delta)$-invariant, we get $|L_{H^p_i}(Y)| \leq \ee^{(\htop(H \curvearrowright Y)+\varepsilon)|H_i|}$ and thus \begin{align*}
	    |L_{F'}(Y_{\gamma}[X])| & \leq |\Sigma_Y|^{\frac{\delta'|F'|}{\delta}} \sum_{p\in L_{S'F'}(X)}\prod_{i\in I(p)}\ee^{(\htop(H \curvearrowright Y)+\varepsilon)|H_i|}\\
	    & \leq  |\Sigma_Y|^{\frac{\delta'|F'|}{\delta}} \sum_{p\in L_{S'F'}(X)}\ee^{(\htop(H \curvearrowright Y)+\varepsilon)|F'|(1-\frac{\delta'}{\delta})}\\
	    & \leq  |\Sigma_Y|^{\frac{\delta'|F'|}{\delta}} \ee^{(\htop(H \curvearrowright Y)+\varepsilon)|F'|} |L_{S'F'}(X)|.
	  \end{align*}
	  
	  Therefore we obtain that, \begin{align*}
	  \frac{1}{|F'|}\log(|L_{F'}(Y_{\gamma}[X])|) & \leq \frac{\delta'}{\delta}\log(|\Sigma_Y|) + (\htop(H \curvearrowright Y)+\varepsilon)+ \frac{1}{|F'|}\log(|L_{S'F'}(X)|)\\
	  & \leq \frac{\delta'}{\delta}\log(|\Sigma_Y|) + \htop(H \curvearrowright Y)+\varepsilon+ \frac{1}{|F'|} \left(\log(|L_{F'}(X)|)+ \log(|L_{S'F'\setminus F'}(X)|) \right) \\
	  & \leq \frac{\delta'}{\delta}\log(|\Sigma_Y|) + \htop(H \curvearrowright Y)+\varepsilon+ \frac{1}{|F'|}\log(|L_{F'}(X)|) + \frac{|S'F'\setminus F'|}{|F'|}\log(|\Sigma_X|).
	  \end{align*}
	  
	  As $F'$ is an $(S',\delta')$-invariant set, we get that $|S'F'\setminus F'| \leq \delta'|F'|$. Furthermore, by definition this also implies that $\log(|L_{F'}(X)| \leq (\htop(G \curvearrowright X) + \varepsilon')|F'|$, therefore for every $(S',\delta')$-invariant set $F'$ we have, \begin{align*}
	  \frac{1}{|F'|}\log(|L_{F'}(Y_{\gamma}[X])|)  \leq \htop(H \curvearrowright Y) + \htop(G \curvearrowright X) +\varepsilon+\varepsilon'+ \frac{\delta'}{\delta}\log(|\Sigma_Y|) + \delta'\log(|\Sigma_X|).
	  \end{align*}
	  
	  By the infimum formula for the entropy, the previous expression is an upper bound for the entropy $h_{\text{top}}(G \curvearrowright Y_{\gamma}[X])$. Now choose $\varepsilon = \varepsilon' = \frac{1}{n}$, this bounds the available values of $\delta$ and $\delta'$ above. We may arbitrarily choose $\delta' \leq \frac{\delta}{n}$. Letting $n$ go to infinity we obtain, \begin{equation} h_{\text{top}}(G \curvearrowright Y_{\gamma}[X]) \leq h_{\text{top}}(H \curvearrowright Y)+ h_{\text{top}}(G \curvearrowright X).\end{equation}
	 
	 For the lower bound,~\Cref{eq_entropyforidiots} shows that a lower bound for $L_{H^p_i}(Y)$ is given by $\ee^{\htop(H \curvearrowright Y)|H^p_i|}$. It is then not hard to see that for every $F' \Subset G$ we have,\begin{align}
	 |L_{F'}(Y_{\gamma}[X])| & \geq |L_{F'}(X)|\ee^{\htop(H \curvearrowright Y)|F'|}.
	 \end{align}
	 From where we obtain the other inequality.
\end{proof}

Recall that by~\Cref{remark_SFTchart_SFT_is_SFT} if both the subshift $X$ in a chart $(X,\gamma)$ and the embedded subshift $Y$ are SFTs, then $Y_{\gamma}[X]$ is also an SFT. This gives us a way of producing a new $G$-SFT those topological entropy is the sum of their entropies.

\begin{corollary}\label{corollary_realize_entropy}
	If $(X,\gamma)$ is a free $G$-chart of $H$ and $X$ is a $G$-SFT, then for every $H$-SFT $Y$ there exists a $G$-SFT $Z$ which has entropy $\htop(G \curvearrowright X)+\htop(H\curvearrowright Y)$. In other words,
	
	\[ \htop(G \curvearrowright X) + \entsft(H) \subset \entsft(G).\]
\end{corollary}

In what follows we shall show that if there is at least one free $G$-chart $(X,\gamma)$ of $H$ where $X$ is a $G$-SFT, then it is always possible to find another such chart where $X$ can have arbitrarily low entropy.

\subsection{Reducing the entropy of a chart}\label{section_reduce_ent_charts}

The goal of this section is to develop a method for reducing the entropy of a subshift of finite type in such a way that the new subshift of finite type preserves any cocycle defined on the original one. In order to do this we will use the machinery of quasitilings developed by Ornstein and Weiss in~\cite{OrnWei1987}. In order to minimize the complexity of the proof, we shall in fact use a recent result by Downarowicz, Huczec and Zhang~\cite{DownarowiczHuczekZhang2019} which shows, for any countable amenable group, the existence of zero-entropy exact tilings where each tile can be made arbitrarily invariant.

The ideas presented in this section have been strongly influenced by the work of Frisch and Tamuz~\cite{FrischTamuz2015} which use similar methods to study generic properties of the set of all subshifts.

\begin{definition}
	Let $G$ be a group. A \define{tile set} is a finite collection $\mathcal{T} = \{T_1,\dots,T_n\}$ of finite subsets of $G$ which contain the identity.
	A \define{tiling} of $G$ by $\mathcal{T}$ is a function $\tau \colon G \to  \mathcal{T} \cup \{\varnothing\}$ such that:
	\begin{enumerate}
		\item ($\tau$ is pairwise-disjoint) For every $g,h \in G$, if $g \neq h$ then $\tau(g)g \cap \tau (h)h =\varnothing$.
		\item ($\tau$ covers $G$) For every $g \in G$ there exists $h \in G$ such that $g \in \tau(h)h$.
	\end{enumerate}
\end{definition}

\begin{lemma}\label{lemma_tilings_are_SFTS}
	Let $\mathcal{T}$ be a tileset. The collection of all tilings of $G$ by $\mathcal{T}$ is a $G$-SFT.
\end{lemma}

\begin{proof}
	Let $X_{\mathcal{T}} \subset (\mathcal{T}\cup \{\varnothing\})^G$ be the set of all configurations $\tau \colon G \to  \mathcal{T} \cup \{\varnothing\}$ which avoid the set of forbidden patterns $\mathcal{D} \cup \mathcal{C}$ where
	\begin{enumerate}
		\item $\mathcal{D}$ is the set of all patterns $p$ with support $\{1_G,g\}$ where $g = t_2^{-1}t_1 \neq 1_G$ for some $t_1, t_2 \in \bigcup_{i \leq n}T_i$ and which satisfy that $p(1_G) \cap p(g)g \neq \varnothing$.
		\item $\mathcal{C}$ consists of all patterns $q$ with support $\bigcup_{i \leq n}T_i^{-1}$ such that $1_G \notin q(g)g$ for every $g \in \supp(q)$.
	\end{enumerate}
	
	Both $\mathcal{D}$ and $\mathcal{C}$ are finite and thus $X_{\mathcal{T}}$ is a $G$-SFT. We claim that $\tau \in X_{\mathcal{T}}$ if and only if $\tau$ is a tiling of $G$ by $\mathcal{T}$. We shall show this in two parts. Let $\tau \in (\mathcal{T}\cup \{\varnothing\})^G$.
	
	\begin{enumerate}
		\item $\tau$ is pairwise disjoint if and only if no pattern from $\mathcal{D}$ appears in $\tau$. Indeed, if $\tau$ is not pairwise disjoint there are $h_1 \neq h_2$ such that $\tau(h_1)h_1 \cap \tau(h_2)h_2 \neq \varnothing$. Letting $\tau' = h_1\tau$ we have $\tau'(1_G) = \tau(h_1)$ and $\tau'(h_2h_1^{-1}) = \tau(h_2)$, therefore $\tau(h_1)h_1 \cap \tau(h_2)h_2 \neq \varnothing$ if and only if $\tau'(1_G) \cap (\tau'(h_2h_1^{-1}))h_2h_1^{-1} \neq \varnothing$. This means that there exist $t_1 \in \tau'(1_G)$ and $t_2 \in \tau'(h_2h_1^{-1})$ such that $t_1 = t_2h_2h_1^{-1}$, equivalently such that $h_2h_1^{-1} = t_2^{-1}t_1$. Let $g = t_2^{-1}t_1$. We get that $\tau'(1_G) \cap (\tau'(g))g \neq \varnothing$ if and only if $\tau'|_{\{1_G,g\}} \in \mathcal{D}$ and thus $\tau'|_{\{1_G,g\}}$ appears in $\tau$.
		\item $\tau$ covers $G$ if and only if no pattern from $\mathcal{C}$ appears in $\tau$. Indeed, suppose $\tau$ does not cover $G$, then there is $g \in G$ such that for every $h \in G$, $g \notin \tau(h)h$. Letting $\tau' = g\tau$ we obtain that $\tau(h) = \tau'(hg^{-1})$ and hence $g \notin \tau(h)h$ for every $h \in G$ if and only if $1_G \notin \tau'(hg^{-1})hg^{-1}$ for every $h \in G$ which is the same as saying that $1_G \notin \tau'(s)s$ for every $s \in G$. This is equivalent to $\tau'|_{\bigcup_{i \leq n}T_i^{-1}} \in \mathcal{C}$. Therefore $\tau$ does not cover $G$ if and only if there is $g \in G$ such that $(g\tau)|_{\bigcup_{i \leq n}T_i^{-1}} \in \mathcal{C}$, which is the same as saying that a pattern from $\mathcal{C}$ appears in $\tau$.
	\end{enumerate}
	Therefore $\tau \in X_{\mathcal{T}}$ if and only if $\tau$ is a tiling of $G$ by $\mathcal{T}$.\end{proof}

\begin{remark}
	The orbit closure of any tiling $\tau \colon G \to  \mathcal{T} \cup \{\varnothing\}$ forms a $G$-subshift which is not necessarily of finite type. We shall denote by $\htop(\tau)$ the topological entropy of said subshift.
\end{remark}

\begin{theorem}[Downarowicz, Huczek and Zhang~\cite{DownarowiczHuczekZhang2019}]\label{teorema_tiling_exacto}
	Let $G$ be a countable amenable group. For any $F\Subset G$ and $\delta >0$ there exists a tile set $\mathcal{T}$ such that every $T \in \mathcal{T}$ is $(F,\delta)$-invariant and there exists a tiling $\tau$ by $\mathcal{T}$ such that $\htop(\tau )=0$.
\end{theorem}

\begin{lemma}\label{lemma_SFT_small}
	Let $G$ be a countable amenable group and $X \subset \Sigma^G$ be a $G$-SFT. Suppose that $Y \subset X$ is a subshift of $X$, then for every $\varepsilon>0$ there exists a $G$-SFT $Z\subset X$ so that \[\htop(G \curvearrowright Y) \leq \htop(G \curvearrowright Z) \leq \htop(G \curvearrowright Y)+\varepsilon.\]
\end{lemma}

\begin{proof}
	Fix $\varepsilon >0$. By~\Cref{eq_entropyforidiots} there exists $D \Subset G$ so that $\log(|L_{D}(Y)| \leq |D|(\htop(G \curvearrowright Y)+\varepsilon)$. Let $\FF_1$ be a set of forbidden patterns which defines $X$ and let $\FF = \FF_1 \cup (\Sigma^D \setminus L_{D}(Y))$. Letting $Z$ be the $G$-SFT defined by the set of forbidden patterns $\FF$, we have $Z \subset X$ and $Y \subset Z$ from where it follows that $\htop(G \curvearrowright Y) \leq \htop(G \curvearrowright Z)$. Furthermore, by construction we get $L_D(Z) = L_D(Y)$ and thus we have \[
	\htop(G \curvearrowright Z) = \inf_{F \in \mathcal{F}(G)} \frac{1}{|F|}\log(|L_F(Z)|)  \leq \frac{1}{|D|}\log(|L_D(Z)| \leq \htop(G \curvearrowright Y)+\varepsilon.\]
	And so $Z$ satisfies the required properties.\end{proof}

Let $T,K$ be finite subsets of $G$. The $K$-core of $T$ is the set $\textrm{Core}_K(T) = \{t\in T \mid Kt \subset T  \}$. It is an easy exercise to show that if $T$ is a $(K,\frac{\delta}{|K|})$-invariant set, then $|T \setminus \textrm{Core}_K(T)|< \delta |T|$, for a proof,  see~\cite[Lemma 2.6]{DownarowiczHuczekZhang2019}.

Now we are ready to state the main theorem of this section which shows that every SFT admits subsystems with arbitrarily low topological entropy and which are also SFTs.

\begin{theorem}\label{theorem_tilings_forthewin}
	Let $G$ be a countable amenable group and $X\subset \Sigma^G$ be a $G$-SFT. For every $\varepsilon > 0$ there exists a $G$-SFT $Z \subset X$ such that $h_{top}(G \curvearrowright Z) \leq \epsilon$
\end{theorem}

\begin{proof}
	
	We claim that it suffices to show that for every $\varepsilon >0$ there exists a $G$-SFT $Y$ (on a different alphabet) such that $h_{top}(G \curvearrowright Y) \leq \epsilon$ and a continuous $G$-equivariant map $\phi\colon Y \to X$. Indeed, if this is the case, using the above result with $\frac{\varepsilon}{2}$ and the property that (for amenable group actions) topological entropy does not increase under topological factor maps, we obtain that $\phi(Y)$ is a subshift of $X$ with entropy $\htop(G \curvearrowright \phi(Y)) \leq \frac{\varepsilon}{2}$. Using~\Cref{lemma_SFT_small} with $\frac{\varepsilon}{2}$ we obtain an SFT $Z \subset X$ whose entropy is bounded by $\htop(G \curvearrowright \phi(Y)) + \frac{\varepsilon}{2} \leq \varepsilon$ as required.
	
	Let us show the above claim. Let $\FF$ be a finite set of forbidden patterns which defines $X$, let $F = \bigcup_{p \in \FF}\supp(p)$ be the union of their supports and $K = FF^{-1}$. By~\Cref{teorema_tiling_exacto} there exists a tileset $\mathcal{T}= \{T_1,\dots,T_n\}$ such that every tile in $\mathcal{T}$ is $(K,\frac{\varepsilon}{4|K|\log(|\Sigma|)})$-invariant and which admits a tiling $\tau^*$ by $\mathcal{T}$ with zero entropy. In particular, the $K$-core of each tile $T \in \mathcal{T}$ satisfies $|T \setminus \textrm{Core}_K(T)|< \frac{\varepsilon}{4\log(|\Sigma|)} |T|$ and we can find a finite set $D \Subset G$ such that $\log(|L_{D}(\overline{ \{g \tau^*\}_{g \in G}}  )|) \leq \frac{\varepsilon}{4} |D|$.
	
	By~\Cref{lemma_tilings_are_SFTS} the set $X_{\mathcal{T}}$ of all tilings of $G$ by $\mathcal{T}$ is a $G$-SFT. Consider the subshift of finite type $X^{\mathcal{L}}_{\mathcal{T}} \subset X_{\mathcal{T}}$ where we additionally forbid the finite set of patterns $\mathcal{L}$:
	
	\[\mathcal{L} = (\mathcal{T}\cup \{\varnothing\})^D \setminus L_{D}(\overline{ \{g \tau^*\}_{g \in G}}  ). \]
	
	Clearly $\tau^* \in X^{\mathcal{L}}_{\mathcal{T}}$, hence $X^{\mathcal{L}}_{\mathcal{T}}$ is a non-empty $G$-SFT. Furthermore we have
	
	\[ \htop(G \curvearrowright X^{\mathcal{L}}_{\mathcal{T}}) = \inf_{F \Subset G} \frac{1}{|F|}\log(|L_{F}(X^{\mathcal{L}}_{\mathcal{T}})|) \leq \frac{1}{|D|}\log(|L_{D}(X^{\mathcal{L}}_{\mathcal{T}})|) \leq  \frac{\varepsilon}{4}.  \]
	
	Consider the set $U \isdef \bigcup_{i \leq n}T_i$. We define $X^{\star}$ as the set of all configurations in $(\Sigma \cup U)^G$ for which no forbidden patterns from $\FF$ appear. Finally, we define $Y \subset X^{\mathcal{L}}_{\mathcal{T}} \times X^{\star}$ as the set of all pairs of configurations $(\tau,x)$ such that for every $g \in G$ if we let $(\tau',x') = (g\tau,gx)$ then we have:
	\begin{enumerate}
		\item If $h \in \textrm{Core}_K(\tau'(1_G))$ then $x'(h) = h$.
		\item If $h \in \tau'(1_G) \setminus \textrm{Core}_K(\tau'(1_G))$ we have $x'(h)\in \Sigma$.
		\item $x'|_{\tau'(1_G)\setminus \textrm{Core}_K(\tau'(1_G))} \in L_{\tau'(1_G)\setminus \textrm{Core}_K(\tau'(1_G))}(X)$.
	\end{enumerate}
	
	In other words, $Y$ is the $G$-subshift which consists of all configurations obtained by overlaying some $x \in X$ with a tiling $\tau \in X^{\mathcal{L}}_{\mathcal{T}}$ and replacing every symbol in the $K$-core of a tile by an address pointing to the center of the tile.
	
	We claim $Y$ is a $G$-SFT. Indeed, it can be obtained from the $G$-SFT $X^{\mathcal{L}}_{\mathcal{T}} \times X^{\star}$ by forbidding the finite collection of all patterns $p$ with support $U$ for which the first coordinate of $p(1_G)$ is some $T \in \mathcal{T}$ and either there is $g \in \textrm{Core}_K(T)$ for which the second coordinate of $p(g)$ is not $g$ or the pattern obtained by restricting the second coordinate of $p$ to ${T \setminus \textrm{Core}_K(T)}$ is not in  $L_{T\setminus \textrm{Core}_K(T)}(X)$. We leave it as an exercise to the reader to verify that $(\tau,x) \in Y$ if and only if no patterns as above appear.
	
	Let us first construct the $G$-equivariant map $\phi \colon Y \to X$. Informally, $\phi$ is the map that erases the tiling $\tau$ and replaces the addresses (which appear in the $K$-core of some $Tg$ for $T \in \mathcal{T}$) by the symbols of some fixed pattern which depends only on the values of $x$ on $Tg \setminus \textrm{Core}_K(T)g$. Formally, associate to every $T \in \mathcal{T}$ and pattern $p \in L_{T\setminus \textrm{Core}_K(T)}(X)$ a pattern $\eta(T,p) \in L_{T}(X)$ such that $\eta(T,p)|_{T\setminus \textrm{Core}_K(T)} = p$. Let $\Phi\colon Y \to \Sigma$ be defined by
	
	\[\Phi(\tau,x) \isdef \begin{cases}
	x(1_G) & \mbox{ if }x(1_G) \in \Sigma \\
	\eta(\tau(h^{-1}),(h^{-1}x)|_{\tau(h^{-1}) \setminus \textrm{Core}_K(\tau(h^{-1}))})(h) & \mbox{ if }x(1_G) = h \in U.
	\end{cases}   \]

	As $U$ is finite this map is local. As a consequence, $\phi \colon Y \to \Sigma^G$ given by $\phi(\tau,x)(g) = \Phi(g\tau,gx)$ is a continuous $G$-equivariant map. 
	
	Let us show that $\phi(\tau,x) \in X$. If it is not the case, then there exists $p \in \FF$ and $g \in G$ such that $\phi(g\tau,gx)|_{\supp(p)} = p$. For simplicity, let us rename $(\tau',x') = (g\tau,gx)$. If for every $s \in \supp(p)$ we have $x'(s) \in \Sigma$ then $\phi(\tau',x')|_{\supp(p)} = x'|_{\supp(p)}$ which cannot be $p$ by definition of $X^{\star}$. Otherwise we have $\bar{s} \in \supp(p)$ such that $x'(\bar{s}) = h \in U$, which in turn means that $\tau'(h^{-1}\bar{s}) \in \mathcal{T}$. In other words, for $f = h^{-1}\bar{s}$ we have $\bar{s} \in \textrm{Core}_K(\tau'(f))f$. By definition of $K$-core, we have that $K\bar{s} \subset \tau'(f)f$. As $\supp(p) \subset F$ and $K = FF^{-1}$ we obtain that $\supp(p)\subset \tau'(f)f$. By definition of $\phi$ and $\eta$ we have that $\phi(f\tau',fx')|_{\tau'(f)} = \eta(\tau'(f), fx'|_{\tau'(f)\setminus \textrm{Core}_K(\tau'(f))}) \in L_{\tau'(f)}(X)$. In particular, $\phi(\tau',x')|_{\tau'(f)f} \in L_{\tau'(f)f}(X)$. As $\supp(p)\subset \tau'(f)f$ this shows that $\phi(\tau',x')|_{\supp(p)} \neq p$, raising a contradiction. 
	
	Lastly, let us verify that $\htop(G\curvearrowright Y) \leq \varepsilon$. As $\htop(G \curvearrowright X^{\mathcal{L}}_{\mathcal{T}}) \leq \frac{\varepsilon}{4}$, we can find $W_1 \Subset G$ and $\delta_1 >0$ such that any $(W_1,\delta_1)$-invariant set $R$ satisfies $\log(|L_{R}(X^{\mathcal{L}}_{\mathcal{T}})|) \leq |R|\frac{\varepsilon}{2}$. Pick $W \isdef W_1 \bigcup U$ and $\delta < \delta_1$ sufficiently small (for instance $\delta < \min(\delta_1,\frac{\varepsilon}{4|U|\log(|\Sigma|)})$) such that any $(W,\delta)$-invariant set $R$ satisfies that $|R \setminus \textrm{Core}_U(R)|< \frac{\varepsilon}{4\log(|\Sigma|)}|R|$. 
	
	Fix $\tau \in X^{\mathcal{L}}_{\mathcal{T}}$ and let us denote by $L_{R}(Y,\tau)$ the set of $p \in L_R(Y)$ for which the first coordinate is $\tau|_{R}$. Let us write $R$ as the disjoint union $R_1 \uplus R_2$ where $R_1$ is the set of all $g \in R$ for which there is $h \in R$ such that $\tau(h)h \subset R$. By definition, as $\tau(h) \subset U$, we have that $R_2 \subset R \setminus \textrm{Core}_U(R)$ and hence $|R_2|< \frac{\varepsilon}{4\log(|\Sigma|}|R|$. On the other hand, the symbols in every position in $\textrm{Core}_K(\tau(h))h$ are fixed. As the $\tau(h)h$ cover $R_1$ and $|\tau(h)h \setminus \textrm{Core}_K(\tau(h)h)|< \frac{\varepsilon}{4\log(|\Sigma|)} |\tau(h)|$ we have at most $\frac{\varepsilon}{4\log(|\Sigma|)}|R_1| \leq \frac{\varepsilon}{4\log(|\Sigma|)}|R|$ positions in $R_1$ are potentially free. Therefore we obtain the bound  \[ |L_{R}(Y,\tau)| \leq |\Sigma|^{|R_2|} |\Sigma|^{\frac{\varepsilon}{4\log(|\Sigma|)}|R_1|} \leq |\Sigma|^{\frac{\varepsilon}{2\log(|\Sigma|)}|R|}.  \]
	Note that this does not depend upon the choice of $\tau$. We can thus obtain
	
	 \[ |L_{R}(Y)| \leq |L_{R}(X^{\mathcal{L}}_{\mathcal{T}})||\Sigma|^{\frac{\varepsilon}{2\log(|\Sigma|)}|R|} \leq  \exp(|R|\frac{\varepsilon}{2})|\Sigma|^{\frac{\varepsilon}{2\log(|\Sigma|)}|R|}.  \]
	
	Therefore \[ \htop(G \curvearrowright Y) \leq \frac{1}{|R|}\log(|L_{R}(Y)|) \leq \frac{1}{|R|} \left( |R|\frac{\varepsilon}{2} + |R|\frac{\varepsilon \log(|\Sigma|)}{2 \log(|\Sigma|)} \right) \leq \varepsilon. \]
	Which completes the proof.
\end{proof}

Before applying~\Cref{theorem_tilings_forthewin} to reduce the entropy of a chart, let us mention a nice application which shows that for any arbitrary countable amenable group, every subshift of finite type must necessarily contain a subsystem with zero topological entropy. This extends the result of Quas and Trow~\cite[Corollary 2.3]{QuasTrow2000} which shows that minimal $\ZZ^d$-SFTs have zero topological entropy and whose argument works for amenable orderable groups. Let us also remark the work of Frisch and Tamuz~\cite{FrischTamuz2015} also gives a way to obtain Quas and Trow's result for arbitrarily countable amenable groups and that the author is aware of a non-published direct proof by Ville Salo which works for any amenable and finitely generated group and relies on a combinatorial argument.

\begin{corollary}
	Let be $G$ a countably infinite amenable group. Any $G$-SFT $X$ contains a $G$-invariant closed subset with zero topological entropy. In particular, every minimal $G$-SFT has zero topological entropy.
\end{corollary}

\begin{proof}
	Let $\varepsilon_n = \frac{1}{n}$ and let $Y_0 = X$. By~\Cref{theorem_tilings_forthewin} there exists a $G$-SFT $Y_1$ such that $\htop(G \curvearrowright Y_1) \leq \varepsilon_1$ and $Y_1 \subset Y_0$. Iterating this procedure we can obtain for every $n \in \NN$ a $G$-SFT $Y_n$ such that $\htop(G \curvearrowright Y_n) \leq \varepsilon_n = \frac{1}{n}$ and $Y_n \subset Y_{n-1}$. As each $Y_n$ is closed we have that $Z = \bigcap_{n \geq 0 }Y_n$ is non-empty. Clearly $Z$ is $G$-invariant as each $Y_n$ is $G$-invariant. Furthermore, $\htop(G\curvearrowright Z)\leq \htop(G \curvearrowright Y_n)$ for every $n \in \NN$, therefore $\htop(G\curvearrowright Z) =0$.
\end{proof}

Let us also remark that this result is in direct contrast with existence of minimal Toeplitz subshifts of arbitrary positive topological entropy on residually finite groups, see~\cite{Cortez2008, Krieger2007_toeptliz, Marthaentropytoeplitz2016}.

To the knowledge of the author, the following question is open even in $\ZZ^2$.

\begin{question}\label{question_zeroentropy}
	Does there exist an amenable group $G$ and a $G$-SFT which does not contain a zero-entropy $G$-SFT?
\end{question}

Let us go back to reducing the entropy of a chart.

\begin{corollary}\label{corollary_reducing_chart_entropy}
	Suppose there exists a free $G$-chart $(X,\gamma)$ for $H$ such that $X$ is a $G$-SFT. Then for every $\varepsilon >0$ there exists a free $G$-chart $(Y,\gamma')$ for $H$ such that $Y$ is a $G$-SFT and $\htop(G \curvearrowright Y) \leq \epsilon$.
\end{corollary}

\begin{proof}
	Apply~\Cref{theorem_tilings_forthewin} to $X$ and $\varepsilon>0$ to obtain a $G$-SFT $Y$ such that $\htop(G \curvearrowright Y) \leq \epsilon$ and $Y \subset X$. Let $\gamma' \colon H \times Y \to G$ be the restriction of $\gamma$ to $Y$. Clearly $\gamma'$ is continuous and an $H$-cocycle.
\end{proof}

\subsection{Conditions for the existence of free charts}\label{section_conditions_charts}

In this section we shall present conditions under which there exist free charts and conditions under which they can be realized with a subshift of finite type. An obvious condition which implies the existence of a free $G$-chart of $H$ is that $H$ embeds into $G$ as a subgroup, see~\Cref{example_obviouschart}. Note that in that case the chart automatically has entropy zero and we obtain the rather obvious corollary that $\entsft(H)\subset \entsft(G)$.

The notion that $H$ embeds into $G$ can be relaxed using the notion of translation-like action introduced by Whyte~\cite{Whyte1999}. We shall see that whenever the groups are finitely generated, this notion is closely related with the existence of free charts.

\begin{definition}
	Let $(X,d)$ be a metric space and $H$ a group. We say that $H \curvearrowright X$ is a \define{translation-like} action if
	\begin{itemize}
		\item $H \curvearrowright X$ is free, that is, for every $x \in X$ then $hx = x$ implies that $h = 1_H$.
		\item $H \curvearrowright X$ is bounded, that is, for every $h \in H$, $\sup_{x \in X} d(x,hx) < \infty$.
	\end{itemize}
\end{definition}

Any finitely generated group $G$ can be seen as a metric space by endowing it with a metric induced by a finite set of generators. In that case, the second condition can be replaced by the condition that for every fixed $h \in H$ the set of all $(h \cdot g)g^{-1}$ is finite.

\begin{proposition}
	Let $H,G$ be finitely generated groups. $H$ acts translation-like on $G$ if and only if there exists a free $G$-chart $(X,\gamma)$ of $H$.
\end{proposition}

\begin{proof}
	Fix a finite set $S$ of generators of $H$. Suppose there exists a translation-like action $H \curvearrowright G$. As the action is bounded and $S$ is finite, we have that the set $F =\{ f \in G \mid (s\cdot g)  = fg \mbox{ for } s\in S, g \in G \}$ is finite. Consider the alphabet $\Sigma = F^S$ and the configuration $x \colon G \to \Sigma$ such that $(x(g))(s) = f \in F$ if and only if $s\cdot g = fg$. Let $X = \overline{\bigcup_{g \in G}\{ gx \}}$ be the orbit closure of $x$. By definition $X$ is a $G$-subshift. For $y \in X$, let $\gamma(s,y) = (y(1_G))(s)$ and extend $\gamma$ to $H \times X$ through the cocycle equation. It is clear that $\gamma$ is continuous. By definition, we have that $s \cdot_x g = \gamma(gx,s)g = (x(g))(s) = (s \cdot g) g^{-1}g = s \cdot g$. In other words, the action $H \overset{x}{\curvearrowright} G$  coincides with $H \curvearrowright G$ and hence it's free. It follows from compactness that the same holds for any $y \in X$ and thus $(X,\gamma)$ is a free $G$-chart of $H$.
	
	Conversely, suppose there exists a free $G$-chart $(X,\gamma)$ of $H$ and let $x \in X$. By definition, the action $H \overset{x}{\curvearrowright} G$ is free. Let $h \in G$, the restriction of $\gamma$ to $\{h\} \times X$ takes finitely many values and depends only on finitely many coordinates of $x \in X$. It follows that $(h \cdot_x g)g^{-1} = \gamma(h,gx)gg^{-1} = \gamma(h,gx)$ takes only finitely many values and hence $H \overset{x}{\curvearrowright} G$ is bounded.
\end{proof}

In other words, the least we can require if we want a free $G$-chart of $H$ is the existence of a translation-like action of $H$ on $G$. In what follows we shall give further conditions under which one can always find a free $G$-chart of $H$ given by a $G$-SFT. The following proof is essentially contained in the work of Jeandel~\cite[Section 2]{Jeandel2015}. 

\begin{proposition}\label{proposition_all_we_need_for_charts}
	Let $H,G$ be finitely generated groups such that:
	
	\begin{enumerate}
		\item $H$ admits a translation-like action on $G$.
		\item $H$ is finitely presented.
		\item There exists a non-empty $H$-SFT for which the $H$-action is free.
	\end{enumerate}
	Then there exists a free $G$-chart $(X,\gamma)$ of $H$ such that $X$ is a non-empty $G$-SFT.
\end{proposition}

\begin{proof}
	The first part of the proof is the same as in the last proposition, let $H\curvearrowright G$ be the translation-like action and suppose $\langle S \mid R \subset S^*\rangle$ is a finite presentation of $H$ where $S = S^{-1}$. By definition, the set  $F =\{ f \in G \mid (s\cdot g)  = fg \mbox{ for } s\in S, g \in G \}$ is finite. Consider the alphabet $\Sigma = F^S$ of all functions from $S$ to $F$ and let $\gamma\colon S^* \times \Sigma^G \to G$ be the map given by $\gamma(s,x)= (x(1_G))(s)$ for $s \in S$ and extended to the free monoid $S^*$ by the condition\[ \gamma(s_1s_2,x) =  \gamma(s_1,\gamma(s_2,x)x)\cdot \gamma(s_2,x) \mbox{ for every $s_1,s_2$ in $S^*$}.\]
	
	Let us first consider the subshift $Y \subset \Sigma^G$ such that for every $s \in S$ and $g \in G$ we have $(y(g))(s) = f$ then $(y(fg))(s^{-1})= f^{-1}$. This is clearly a subshift of finite type. Let us note that for $y \in Y$, $g \in G$ and $s \in S$ we have,\begin{align*}\gamma(s^{-1}s,gy) & =  \gamma(s^{-1},\gamma(s,gy)gy)\cdot \gamma(s,gy)\\
	& = \gamma(s^{-1},[(y(g))(s)]gy) \cdot (y(g))(s)\\
	& = y([(y(g))(s)](s^{-1}) \cdot (y(g))(s) = 1_G.
	\end{align*}
	The same holds for $\gamma(ss^{-1},gy)$. By a similar argument, it can be shown that if $w \in S^*$ is a word that can be freely reduced to the identity, then $\gamma(w,gy)=1_G$ for every $g \in G$. In other words, $(Y,\gamma)$ codes the free group on $S$ generators.
	
	Let us define $X \subset Y$ as the set of all configurations $x \in Y$ such that whenever $s_1s_2\dots s_{n-1}s_n \in R$ then for every $g \in G$, if we define $f_1 = (y(g))(s_n)$, $f_2 = (y(f_1g))(s_{n-1})$ and for every $k \leq n$,
	 \[f_k = (y(f_{k-1}\dots f_1g))(s_{n+1-k}).\]  Then we have $f_nf_{n-1}\dots f_1 = 1_G$. As $R$ is finite, these conditions can be imposed by forbidding patterns with support bounded by $F^{n}$. In other words, $X$ is also a $G$-subshift of finite type. Again, by the previous calculation, we obtain that for every $w \in R$ and $g \in G$ we have $\gamma(w,gx) = 1_G$ Moreover, as every word which represents $1_G$ in $G$ can be obtained by freely conjugating and concatenating words in $R$, we have that any word $w \in S^*$ which represents the identity satisfies $\gamma(w,gx) = 1_G$. In other words, $(X,\gamma)$ codes a $G$-chart of $H$.
	
	It is not true that $(X,\gamma)$ is free. In fact, the configuration such that $(x(g))(s) = 1_G$ belongs to $X$. However, the configuration $\bar{x}\in X$ defined using the free action $H \curvearrowright G$ by $(\bar{x}(g))(s) = (s \cdot g)g^{-1}$ satisfies that $H \overset{x}{\curvearrowright} G$ = $H \curvearrowright G$. 
	
	By hypothesis, there exists an $H$-subshift $Z$ on which $H$ acts freely. Let us consider $Z_{\gamma}[X]$. By~\Cref{remark_chart_having_one_free_action_is_good} we have that $Z_{\gamma}[X]$ is non-empty. Let $\widehat{\gamma} \colon H \times  Z_{\gamma}[X] \to G$ be the map defined by $\widehat{\gamma}(h,(z,x)) = \gamma(h,x)$. We claim the $G$-chart $(Z_{\gamma}[X], \widehat{\gamma})$ of $H$ is free.
	
	Indeed, if it is not free, there is $(z,x) \in Z_{\gamma}[X]$ and $h\neq 1_H$ such that $h \cdot_{(z,x)} g = g$. Equivalently, such that $\gamma(h,gx) = 1_G$ or $h \cdot_x g = g$. Hence, we would have that \[h\pi_{x,g}(z) = \pi_{x, h \cdot_{x} g}(z) = \pi_{x,g}(z).\]
	
	As $\pi_{x,g}(z) \in Z$, this gives a configuration for which the shift does not act freely, which contradicts the assumption on $Z$.\end{proof}

Let us gather all our results in a single theorem for further reference.

\begin{theorem}\label{theorem_HG} Let $G,H$ be finitely generated amenable groups. Suppose that 
	\begin{enumerate}
		\item $H$ admits a translation-like action on $G$.
		\item $H$ is finitely presented.
		\item There exists a non-empty $H$-SFT for which the $H$-action is free.
	\end{enumerate}
	Then, for every $\varepsilon >0$ there exists a $G$-SFT $X$ such that $h_{top}(G\curvearrowright X) < \varepsilon$ and \[h_{top}(G\curvearrowright X)+ \entsft(H) \subset \entsft(G).\]
\end{theorem}

\begin{proof}
	By~\Cref{proposition_all_we_need_for_charts} there exists a free $G$-chart $(X,\gamma)$ of $H$ such that $X$ is a $G$-SFT. Furthermore, by~\Cref{corollary_reducing_chart_entropy} we can choose it so that $h_{top}(G\curvearrowright X) < \varepsilon$. Finally, we conclude by applying~\Cref{corollary_realize_entropy}.
\end{proof}

\section{Characterization of entropies: the case $H = \ZZ^2$}\label{section_characterization_Z2}

The goal of this section is to exploit~\Cref{theorem_HG} for the case $H = \ZZ^2$. The interest on this particular case comes from the fact we already have a full characterization of the entropies of $\ZZ^2$-SFTs by~\Cref{theorem_HochmanTom}. Furthermore, $\ZZ^2 \cong \langle a,b \mid aba^{-1}b^{-1} \rangle$ is finitely presented, and there exist non-empty $\ZZ^2$-SFTs for which the $\ZZ^2$-action is free, for instance the Robinson tiling~\cite{Robinson1971}.

There is a single obstacle that stops us from getting a characterization for all groups on which $\ZZ^2$ acts translation-like: even if we can choose the entropy of the chart to be arbitrarily low, there is no guarantee that said entropy will be an upper semi-computable number. In what follows we shall show that this is indeed the case if $G$ is a finitely generated group with decidable word problem.

Given a set $S \subset G$ denote by $S^*$ the formal set of all finite words $s_1s_2\dots s_n \in S^*$. Also, for any such word in $S^*$ denote by $\underline{s_1s_2\dots s_n}$ the unique element of $G$ represented by it.

\begin{definition}
	Let $G$ be a finitely generated group and $S$ a finite set of generators. The \define{word problem} of $G$ is the set of all words over the alphabet $S$ which represent the identity of $G$. \[\texttt{WP}_S(G) =\{ w \in S^* \mid \underline{w} = 1_G  \}. \]
\end{definition}

We say that $G$ has \define{decidable word problem} if the language $\texttt{WP}_S(G)$ is decidable for some finite set of generators $S$. It can be shown that this notion is independent of the chosen set of generators and thus, modulo many-one equivalence, one can speak about the \define{word problem} $\texttt{WP}(G)$ of $G$ without making reference to a specific set of generators.

We shall also need to introduce the set of locally admissible patterns.

\begin{definition}
	Let $\Sigma$ be a finite alphabet and $\FF$ be a list of forbidden patterns which defines a subshift $X_{\FF}$. For $F\Subset G$ We say that $q \in \Sigma^F$ is in the set of \define{locally admissible patterns}  $L_{F}^{\texttt{loc}}(X_{\FF})$ if no patterns from $\FF$ appear in $q$, namely, $[q] \not\subset g[p]$ for every $g \in G$ and $p \in \FF$.
\end{definition}

\begin{lemma}\label{lemma_aproximalito}
	Let $G$ be a countable group and $X_{\FF} \subset \Sigma^G$ be a subshift defined by a set of forbidden patterns $\FF$. For any $F\Subset G$ there exists $K \Subset G$ such that $K \supset F$ and $p \in L_{F}(X)$ if and only if there exists $q \in L^{\texttt{loc}}_{K}(X)$ such that $q|_F = p$. 
\end{lemma}

\begin{proof}
	If $G$ is finite the result is obvious. Otherwise we may fix an enumeration $\{g_n\}_{n \in \NN}$ of $G$, let $F^n = F \cup \bigcup_{k \leq n}\{g_k\}$ and consider $p \in \Sigma^F \setminus L_{F}(X)$. We claim there must exist an integer $n(p)$ such that $q|_F \neq p$ for every $q \in  L^{\texttt{loc}}_{F^{n(p)}}(X)$. If this was not the case, we may choose for every $n$ a pattern $q^n \in L^{\texttt{loc}}_{F^{n}}(X)$ such that $q|_F = p$. As the sequence of $[q_n] \subset [p]$ is closed and nested, the intersection $Y = \bigcap_{n \in \NN} [q_n]$ is non-empty and $Y \subset [p]$, and any configuration $y \in Y$ satisfies that no forbidden patterns appear, hence $y \in X \cap [p]$ and thus $p \in L_F(X)$.
	
	As $\Sigma^F$ is finite we may define $N \isdef \max_{p \in \Sigma^F \setminus L_{F}(X)}n(p)$ and $K \isdef F^{N}$. By definition of $N$, we have that if $p \in \Sigma^F \setminus L_{F}(X)$ then $q|_F \neq p$ for every $q \in L^{\texttt{loc}}_{K}(X)$. Conversely, if $p \in L_F(x)$ there exists $x$ such that $x|_F = p$. Defining $q \isdef x|_{K}$ we have $q|_{F}=p$ and $q \in L_{K}(X)\subset L^{\texttt{loc}}_{K}(X)$.\end{proof}

In what follows we shall need to briefly introduce the notion of pattern codings and effectively closed subshifts in finitely generated groups. An introduction to this topic can be found on~\cite{ABS2017}.

\begin{definition}
	Let $G$ be a finitely generated group, $S$ a finite set of generators and $\Sigma$ an alphabet. A function $c\colon W \to \Sigma$ from a finite subset $W$ of $S^*$ is called a \define{pattern coding}. The cylinder defined by a pattern coding $c$ is given by
	\[ [c] = \bigcap_{w \in W} \underline{w}[c(w)].  \]
\end{definition}

In other words, a pattern coding is a coloring of a finite subset of the free monoid $S^*$. A set $\mathcal{C}$ of pattern codings defines a $G$-subshift $X_{\mathcal{C}}$ by setting \[X_{\mathcal{C}} = \Sigma^G \setminus \bigcup_{g \in G, c \in \mathcal{C}} g[c]. \]

We say that a $G$-subshift $X$ is \define{effectively closed} if there exists a recursively enumerable set of pattern codings $\mathcal{C}$ such that $X = X_{\mathcal{C}}$. Obviously, every $G$-SFT is effectively closed.

We shall need the following result.

\begin{lemma}[Lemma 2.3 of~\cite{ABS2017}]\label{lemma_ABS}
	Let $G$ be a finitely generated and recursively presented group. For every effectively closed subshift $X \subset \Sigma^G$ the maximal --for inclusion-- set of forbidden pattern codings that defines $X$ is recursively enumerable.
\end{lemma}

\begin{proposition}\label{proposition_ECSubshift_has_USC_entropy}
	Let $G$ be a finitely generated amenable group with decidable word problem. For every effectively closed subshift $X \subset \Sigma^G$ the topological entropy $h_{\text{top}}(G \curvearrowright X)$ is upper semi-computable.
\end{proposition}

\begin{proof} 
	Let us fix a symmetric set $S$ of generators for $G$. We shall first define three algorithms $T_{\texttt{WP}},T_{\texttt{pat}},T_{\texttt{color}}$ which will be used in the proof.
	
	First, as $G$ has decidable word problem there is an algorithm $T_{\texttt{WP}}$ which on input $w \in S^*$ halts and accepts if and only if $\underline{w}=1_G$. 
	
	Second, as $X$ is effectively closed, by~\Cref{lemma_ABS} there exists a maximal recursively enumerable set of pattern codings $\C^*$ such that $X = X_{\C^*}$. We define $T_{\texttt{pat}}$ as the algorithm which on input $n \in \NN$ yields the list of the first $n$ pattern codings $[c_1,c_2,\dots,c_n]$ of $\C^*$.
	
	Finally, let us denote by $\equiv_{n}$ the equivalence relation on $\bigcup_{k \leq n}S^k$ defined by $u \equiv_{n} v$ if and only if $T_{\texttt{WP}}$ accepts $uv^{-1}$. Let $B_n \isdef \bigcup_{k \leq n}S^k / \equiv_{n}$. We define $T_{\texttt{color}}$ as the algorithm, which on input $n \in \NN$ computes the set of all functions $x\colon B_n \to \Sigma$ such that for every pattern coding $c_i \colon W_i \to \Sigma$ listed by $T_{\texttt{pat}}$ on input $n$ we have that either $W_i \setminus B_n \neq \emptyset$ or $x(w) \neq c(w)$ for at least one $w \in W_i$.
	
	In simpler words, $T_{\texttt{color}}$ enumerates all patterns over a representation of the ball of size $n$ of the Cayley graph of $G$ where the first $n$ forbidden pattern codings do not appear at the identity.
	
	Now we construct an algorithm $T_{\texttt{ent}}$ which on input $n$ outputs a rational number $h_n$ as follows. First apply algorithm $T_{\texttt{color}}$ on input $n$ to produce a set $\{x_1,\dots, x_{M(n)}\}$ of colorings as above.  For each $A\subset B_n$ we define $L^A_n$ as the set of restrictions $\{x_1|_A,\dots, x_{M(n)}|_A\}$ to $A$. Let us define $h^A_n$ as the smallest rational number of the form $\frac{k}{2^n}$ such that \[\frac{1}{|A|}\log(|L^A_n|) < \frac{k}{2^n}.\]
	
	Finally, let us define $h_n \isdef \min_{A \subset B_n}\{h^A_n\}$. From the above definitions, it is clear that each $h_n$ can be computed in a finite number of steps with $T_{\texttt{ent}}$. We claim that the sequence $\{h_n\}_{n \in \NN}$ is non-increasing and that $\inf_{n \in \NN}h_n = h_{\text{top}}(G \curvearrowright X)$.
	
	Indeed, let $m > n$. Clearly for $A \subset B_n$ we have $L^A_m \subset L^A_n$, hence $|L^A_n| >  |L^A_m|$ hence we obtain
	
	\[h_{m} = \min_{A \subset B_{m}}\{h^A_{m}\} \leq \min_{A \subset B_n}\{h^A_{m}\} \leq \min_{A \subset B_n}\{h^A_{n}\} = h_n.  \]
	Hence the sequence $\{h_n\}_{n \in \NN}$ is non-increasing. It is clear from the definition that for every $n \in \NN$ such that $B_n \supset A$ we have $L^A_n \supset L_A(X)$, hence $h^A_{n} > \frac{1}{|A|}\log(|L^A_n|) \geq \frac{1}{|A|}\log(|L^A(X)|)$ and thus by~\Cref{eq_entropyforidiots}, \[h_n > \inf_{A \subset B_n}\frac{1}{|A|}\log(|L_A(X)|) \geq h_{\text{top}}(G \curvearrowright X).\]
		
	Similarly, by~\Cref{eq_entropyforidiots} for every $\varepsilon >0$ there exists a fixed finite $F \subset G$ such that $\frac{1}{|F|}\log(|L_F(X)|)-h_{\text{top}}(G \curvearrowright X) \leq \epsilon$. By~\Cref{lemma_aproximalito} there exists $K$ such that $p \in L_{F}(X)$ if and only if there exists $q \in L^{\texttt{loc}}_{K}(X)$ such that $q|_F = p$. Choose $N_1$ such that $B_{N_1} \supset K$ and $N_2$ so that all pattern codings of $\C^*$ whose support is contained in $K$ have already appeared. Let $N \geq \max(N_1,N_2)$. By definition we have that $L^{K}_{N}= L^{\texttt{loc}}_{K}(X)$ and thus $L^{F}_{N} = L_F(X)$, hence we have that \[h_N \leq h^F_N \leq \frac{1}{|F|}\log(|L_F(X)|)+\frac{1}{2^N} \leq h_{\text{top}}(G \curvearrowright X) + \epsilon+\frac{1}{2^N}. \]
	The last inequality shows that $\{h_n\}_{n \in \NN}$ converges to $h_{\text{top}}(G \curvearrowright X)$.\end{proof}

From this, we can obtain the following characterization.

\begin{theorem}\label{theorem_caract_entropies_G_z2_translation_like}
	Let $G$ be a finitely generated amenable group with decidable word problem which admits a translation-like action by $\ZZ^2$. The set of entropies attainable by $G$-subshifts of finite type is the set of non-negative upper semi-computable numbers.
\end{theorem}

\begin{proof}
	By hypothesis there exists a translation-like action of $\ZZ^2$ on $G$. Therefore $\ZZ^2,G$ satisfy the hypothesis of~\Cref{theorem_HG} which means that for every $\varepsilon>0$ there exists a $G$-SFT $X$ such that $h_{top}(G\curvearrowright X) < \varepsilon$ and \[h_{top}(G\curvearrowright X)+ \entsft(\ZZ^2) \subset \entsft(G).\]
	
	Recall that by~\Cref{theorem_HochmanTom} $\entsft(\ZZ^2)$ is precisely the set of non-negative upper semi-computable real numbers. As $G$ has decidable word problem,~\Cref{proposition_ECSubshift_has_USC_entropy} implies that $\entsft(G) \subset \entsft(\ZZ^2)$. Noting that $0 \in \entsft(G)$ and that the set of upper semi-computable numbers is stable under addition, if we let $\varepsilon$ go to zero we obtain
	\[\entsft(G) = \entsft(\ZZ^2).\]
	Which is what we wanted to show\end{proof}

\section{Consequences}\label{section_consequences}

In the remainder of this section we shall make use of the following simple construction.

\begin{definition}
	Let $H \leq G$ be a subgroup, $\{0\}$ be the trivial $G$-subshift with one point and let the $H$-cocycle $\gamma\colon H \times \{0\} \to G$ be the canonical free $G$-chart of $H$ defined by $\gamma(h,0) = h$. 
	
	For an $H$-subshift $X$ denote by $X^{\uparrow G}$ the \define{free $G$-extension of $X$} defined by $X_{\gamma}[\{0\}]$. 
\end{definition}

\begin{proposition}\label{proposition_same_entropy_free_subshift}
	Let $G$ be a countable amenable group, $H \leq G$ and $X$ be an $H$-subshift. Then	\[h_{\text{top}}(G \curvearrowright X^{\uparrow G}) = h_{\text{top}}(H \curvearrowright X).\]
\end{proposition}

\begin{proof}
	By~\Cref{theorem_addition_formula} we have \[h_{\text{top}}(G \curvearrowright X^{\uparrow G}) = h_{\text{top}}(H \curvearrowright X) + h_{\text{top}}(G \curvearrowright \{0\}) = h_{\text{top}}(H \curvearrowright X).\]Which is what we wanted to show.\end{proof}

We shall also need the following result which relates the entropies of subshifts of finite type in a group to those of a finite index subgroup.

\begin{lemma}\label{proposition_virt_Z_has_perronentropies}
	Let $G$ be a countable amenable group and let $H \leq G$ be a finite index subgroup. Assume that $\entsft(H)$ is closed under division by positive integers. Then $\entsft(G) = \entsft(H)$.
\end{lemma}

\begin{proof}
	For any $H$-SFT $X$ we can consider the $G$-SFT $X^{\uparrow}$. By~\Cref{proposition_same_entropy_free_subshift} we get $\entsft(H)\subset \entsft(G)$. 
	
	For the converse, let $Y \subset \Sigma^G$ be a $G$-SFT and consider $H \curvearrowright Y$ the restriction of the $G$ action on $Y$ to $H$. It is a well known property of topological entropy that $\frac{1}{[G:H]}\htop(H \curvearrowright Y) =  \htop(G \curvearrowright Y)$. It suffices to show that $H \curvearrowright Y$ is conjugated to an $H$-SFT. Indeed, as $\entsft(H)$ is closed under division by positive integers the above formula yields the result.
	
	Choose a set $R$ of left representatives of $G/H$ and define the $R$-higher power shift $X^{[R]}$ by \[X^{[R]} = \{ x \in (\Sigma^R)^H \mid \exists y \in Y, \mbox{ for every } r \in R, h \in H, (x(h))(r) = y(rh)    \}.    \]
	
	As $R$ is finite, it is clear that $X^{[R]}$ is closed and $H$-invariant and hence that it is an $H$-subshift. The function $\phi\colon X^{[R]} \to Y$ that sends $x \mapsto y$ by $\phi(x)(rh) = (x(h))(r)$ is clearly a continuous bijection. It is also $H$-equivariant:	\[ h'\phi(x)(rh) = \phi(x)(rhh') = (x(hh'))(r) = (h'x(h))(r) = \phi(h'x)(rh). \]
	Therefore $H \curvearrowright X^{[R]}$ is conjugated to $H \curvearrowright Y$. The construction of the forbidden patterns that show that $X^{[R]}$ is an $H$-SFT whenever $Y$ is a $G$-SFT is a simple exercise. The reader may find it in either~\cite[Definition 3.1]{CarrollPenland} or in~\cite[Proposition 9.3.33]{AubBarJea2018}.
\end{proof}

\begin{question}
	Is there any infinite and finitely generated amenable group $G$ for which $\entsft(G)$ is not closed under division by positive integers? 
\end{question}

\subsection{Polycyclic-by-finite groups}\label{subsec:poly}

The goal of this section is to give a full characterization of the set of real numbers attainable as entropies of subshifts of finite type on a polycyclic-by-finite group. In what follows we shall introduce polycyclic groups and state a few of their properties. A good reference is~\cite{Seg2005} or~\cite{DruKap2018book}.

A group $G$ is called \define{polycyclic} if there exists a finite sequence of subgroups \[G = N_1 \triangleright N_2 \triangleright \dots \triangleright N_{n} \triangleright N_{n+1} = \{1_G\}.\]
such that every quotient $N_{i}/N_{i+1}$ is cyclic. The number of $i$ such that $N_{i}/N_{i+1}$ is infinite does not depend on the choice of sequence and is thus a group invariant called the \define{Hirsch index} of $G$ and denoted by $h(G)$.

If we replace the condition that each $N_{i}/N_{i+1}$ is cyclic by the condition that each $N_{i}/N_{i+1}$ is the infinite cyclic group, we obtain the class of \define{poly-$C_{\infty}$} groups. There are polycyclic groups which are not poly-$C_{\infty}$, for instance any cyclic finite group. However, they are very close in the following sense. A proof can be found in either of the two references mentioned above.

\begin{proposition}
	The following are equivalent:
	\begin{enumerate}
		\item $G$ is virtually polycyclic.
		\item $G$ is polycyclic-by-finite.
		\item $G$ is poly-$C_{\infty}$-by-finite.
	\end{enumerate}
\end{proposition}

In particular, as every short sequence $1 \to N \to G \to \ZZ \to 1$ splits, the last proposition means that any virtually polycyclic group can be written as a series $G = N_0 \triangleright N_1 \triangleright \dots \triangleright N_{n} \triangleright N_{n+1} = \{1_G\}$ such that for $i \geq 1$ we have $N_{i} = N_{i+1} \rtimes \ZZ$ and $G$ is virtually $N_1$. Moreover if this is the case then $h(G)=n$.

\begin{theorem}\label{theorem_polycyclic}
	Let $G$ be a virtually polycyclic group. Then 
	\begin{enumerate}
		\item If $h(G)=0$ then $\entsft(G) = \{ \frac{1}{|G|}\log(n) \mid n \in \ZZ_{+}\}$.
		\item If $h(G)=1$ then $\entsft(G) = \entsft(\ZZ)$, the set of non-negative rational multiples of logarithms of Perron eigenvalues.
		\item If $h(G)\geq 2$ then $\entsft(G) = \entsft(\ZZ^2)$, the set of non-negative upper semi-computable numbers.
	\end{enumerate}
\end{theorem}

\begin{proof}
	As $G$ is poly-$C_{\infty}$-by-finite, we have that $G = N_0 \triangleright N_1 \triangleright \dots \triangleright N_{h(G)} \triangleright N_{h(G)+1} = \{1_G\}$ where every quotient except the first one is an infinite cyclic group. If $h(G) = 0$, then $G = N_0 \triangleright N_1 = \{1_G\}$ is necessarily a finite group $F$. As every F\o lner sequence in a finite group is eventually the whole group, we have that for any subshift $X \subset \Sigma^F$,
	
	\[\htop(F \curvearrowright X) = \frac{1}{|F|}\log(|L_F(X)|).\]
	
	In particular, the entropy of every subshift is of the form we claim. To show that every such number occurs, consider the SFT $X^n_{\textrm{unif}} \subset \{1,2,\dots,n\}^{F}$ consisting of the uniform configurations $x_i$ such that $x_i(f) = i$ for every $f \in F$. Clearly $\htop(F \curvearrowright X^n_{\textrm{unif}}) = \frac{1}{|F|}\log(n)$. This proves the first claim.
	
	If $h(G)=1$ then $G = N_0 \triangleright N_1 \triangleright N_{2} = \{1_G\}$. As $N_1 \cong \{1_G\}\rtimes \ZZ$ then $N_1 \cong \ZZ$. This means that $G$ is virtually $\ZZ$. By~\Cref{proposition_virt_Z_has_perronentropies} the claim holds for this case as well.
	
	Let $h(G)\geq 2$. We will show that $\ZZ^2$ embeds into $G$. Indeed, we have that $N_{h(G)}\cong \ZZ$ and that $N_{h(G)-1} \cong N_{h(G)} \rtimes \ZZ$ is a subgroup of $G$. Hence, have that $N_{h(G)-1} \cong \ZZ \rtimes_{\varphi} \ZZ$ for some homomorphism $\varphi\colon \ZZ \to \textrm{Aut}(\ZZ)$. There are two cases: either $\varphi(1) = \textrm{id}$ or $\varphi(1)$ is multiplication by $-1$. The first case yields $N_{h(G)-1} \cong \ZZ^2$ and hence $\ZZ^2$ embeds into $G$. In the second case note that $\varphi(2)=\textrm{id}$ and thus $\ZZ \rtimes_{\varphi} 2\ZZ$ is isomorphic to $\ZZ^2$. Hence $N_{h(G)-1}$ contains a finite index copy of $\ZZ^2$ and thus as $N_{h(G)-1}$ embeds into $G$, we obtain that $\ZZ^2$ embeds into $G$ as well.
	
	Therefore, whenever $h(G)\geq 2$ we have that $\ZZ^2$ embeds into $G$. In particular $\ZZ^2$ acts translation-like on $G$. As every polycyclic-by-finite group is finitely generated and has decidable word problem, we can apply~\Cref{theorem_caract_entropies_G_z2_translation_like} to obtain the desired conclusion.
\end{proof}

\begin{remark}\label{rem:notfullpower}
	In the previous proof we did not use the full power of~\Cref{theorem_caract_entropies_G_z2_translation_like}. We only applied it to the case where $\ZZ^2$ actually embeds into $G$. The next application will rely strongly on translation-like actions.
\end{remark}

\subsection{Products of infinite finitely generated groups}\label{subsec:products}

In this section we shall make use of the following theorem by Seward~\cite{Seward2014}

\begin{theorem}[Theorem 1.4 of~\cite{Seward2014}]\label{theorem_ofSeward}
	Every infinite and finitely generated group admits a translation-like action of $\ZZ$.
\end{theorem}

\begin{corollary}\label{corollary_ofsewards}
	Let $G_1,G_2$ be infinite and finitely generated groups. Then $G_1 \times G_2$ admits a translation-like action of $\ZZ^2$.
\end{corollary}

\begin{proof}
	By~\Cref{theorem_ofSeward}, there exist translation-like actions $\ZZ \overset{\alpha_1}{\curvearrowright} G_1$ and $\ZZ \overset{\alpha_2}{\curvearrowright} G_2$. The $\ZZ^2$-action given by $(n_1,n_2) \cdot (g_1,g_2) \isdef (n_1\cdot_{\alpha_1}g_1,n_2\cdot_{\alpha_2}g_2)$ satisfies the requirements.
\end{proof}

\begin{corollary}\label{corollary_entropy_ofproducts}
	Let $G_1,G_2$ be two infinite, amenable and finitely generated groups with decidable word problem. The set of topological entropies of non-empty $G_1 \times G_2$-SFTs is exactly the set of non-negative upper semi-computable numbers.
\end{corollary}

\begin{proof}
	Clearly $G_1 \times G_2$ has decidable word problem. By the previous corollary it admits a translation-like action of $\ZZ^2$. The result follows from~\Cref{theorem_caract_entropies_G_z2_translation_like}.
\end{proof}

\subsection{Countably infinite amenable groups}\label{subsec:product_nonfg}

Let us now consider the case of countably infinite amenable groups which are not necessarily finitely generated. In the remainder of this section we will need to speak about the word problem for arbitrary countable groups. We shall say that a group presentation $\langle \NN \mid R \subset \NN^* \rangle$ has decidable word problem if there exists an algorithm which on entry $w \in \NN^*$ decides whether $\underline{w} = 1$ in the group defined by that presentation. We shall say that a countable group $G$ has \define{decidable word problem} if it admits a presentation with decidable word problem. Note that if $G$ has decidable word problem, then every finitely generated subgroup of $G$ also does, but the converse may not hold, see for instance~\cite[Example 5.4]{Barbieri2017Tesis}.

\begin{proposition}\label{proposition_countablegrouphasUSCentropy}
	Let $G$ be a countably infinite amenable group which admits a decidable presentation and let $X \subset \Sigma^G$ be a $G$-subshift of finite type. Then $\htop(G \curvearrowright X)$ is upper-semi computable.
\end{proposition}

\begin{proof}
	If $X$ is a $G$-subshift of finite type, there is a finite set of patterns $\FF$ which defines it. Let $S = \bigcup_{p \in \FF}\supp(p)$ be the union of the supports of patterns in $\FF$ and let $H = \langle S \rangle \leq G$ be the finitely generated subgroup of $G$ generated by $S$. As $G$ is amenable and has decidable word problem, then $H$ is amenable and has decidable word problem. Let $Y$ be the $H$-subshift defined by $\FF$. We clearly have that $X = Y^{\uparrow G}$ where $Y^{\uparrow G}$ is the free $G$-extension of $Y$. By~\Cref{proposition_ECSubshift_has_USC_entropy} $h_{\text{top}}(H \curvearrowright Y)$ is upper semi-computable. Therefore by~\Cref{proposition_same_entropy_free_subshift} we have that $h_{\text{top}}(H \curvearrowright Y) =h_{\text{top}}(G \curvearrowright Y^{\uparrow G}) = h_{\text{top}}(G \curvearrowright X)$ and hence $h_{\text{top}}(G \curvearrowright X)$ is also upper semi-computable.
\end{proof}

\begin{corollary}~\label{corollary_caract_entropies_full}
	Let $G$ be an amenable countably infinite group with decidable word problem and which admits a finitely generated subgroup on which $\ZZ^2$ acts translation-like. Then
	\[ \entsft(G) = \entsft(\ZZ^2). \]
\end{corollary}

\begin{proof}
	By~\Cref{proposition_countablegrouphasUSCentropy} we get $\entsft(G) \subset \entsft(\ZZ^2)$. Let $H$ be a finitely generated subgroup on which $\ZZ^2$ acts translation-like. As $H$ has decidable word problem and is amenable, by~\Cref{theorem_caract_entropies_G_z2_translation_like} $\entsft(H)= \entsft(\ZZ^2)$. For any $r \in  \entsft(H)$, there is an $H$-SFT $X$ such that $\htop(H \curvearrowright X)= r$. By~\Cref{proposition_same_entropy_free_subshift} we have $\htop(G \curvearrowright X^{\uparrow}) = r$ and hence $\entsft(H) \subset \entsft(G)$. This gives $\entsft(G) = \entsft(\ZZ^2)$.
\end{proof}

\begin{corollary}~\label{corollary_caract_entropies_full2}
	Let $G_1,G_2$ be amenable, countably infinite and non-locally finite groups with decidable word problem. Then
	\[ \entsft(G_1 \times G_2) = \entsft(\ZZ^2). \]
\end{corollary}

\begin{proof}
	$G_1\times G_2$ is amenable, countably infinite and has decidable word problem. Furthermore, as neither group is locally finite, there are infinite and finitely generated subgroups $H_1 \leq G_1$ and $H_2 \leq G_2$. By~\Cref{corollary_ofsewards} $H_1 \times H_2$ admits a translation-like action of $\ZZ^2$. The result follows from~\Cref{corollary_caract_entropies_full}.
\end{proof}

\begin{remark}
	The non-locally finite condition in~\Cref{corollary_caract_entropies_full2} is necessary. If $G$ is a locally finite group and $X \subset \Sigma^G$ is a subshift of finite type. We can use the same technique as in~\Cref{proposition_countablegrouphasUSCentropy} to reduce its entropy to the entropy of the group which is finitely generated by the support of its forbidden patterns. But the entropy of any subshift in a finite group is necessarily a rational multiple of the logarithm of a positive integer.
\end{remark}

\subsection{Branch groups}\label{subsec:branch}

Suppose that $G$ is a countable amenable group with decidable word problem which contains the product of two non-locally finite and countably infinite subgroups $G_1 \times G_2$ as a subgroup. Then~\Cref{corollary_caract_entropies_full2} and~\Cref{corollary_caract_entropies_full} imply that $\entsft(G) = \entsft(\ZZ^2)$. 

There are many examples satisfying the previous hypothesis within the class of branch groups~\cite{BartholdiGrigorchuk2003Branchgroups}. There is more than one definition of branch group, we shall work with the following one:

\begin{definition}
	A group $G$ is called a \define{branch group} if there exist two sequences of groups $(L_i)_{i \in \NN}$ and $(H_i)_{i \in \NN}$ and a sequence of positive integers $(k_i)_{i \in \NN}$ such that $k_0 = 1$, $G = L_0 = H_0$ and:
	\begin{enumerate}
		\item $\bigcap_{i \in \NN}{H_i} = 1_G$.
		\item $H_i$ is normal in $G$ and has finite index.
		\item there are subgroups $L_i^{(1)},\dots, L_i^{k(i)}$ of $G$ such that $H_i = L_i^{(1)}\times\dots\times L_i^{k(i)}$ and each of the $L_i^{(j)}$ is isomorphic to $L_i$.
		\item Conjugation by elements of $g$ transitively permutes the factors in the above product decomposition.
		\item $k_{i}$ properly divides $k_{i+1}$ and each of the factors $L_i^{(j)}$ contains $k_{i+1}/k_i$ factors $L_{i+1}^{(j')}$.
	\end{enumerate}
\end{definition}

This allows us to state the following result

\begin{theorem}\label{theorem_branch_groups}
	Let $G$ be an infinite, finitely generated, amenable branch group with decidable word problem. Then $\entsft(G) = \entsft(\ZZ^2)$.
\end{theorem}

\begin{proof}
	By the fifth property above, $k_1 > 1$. Furthermore, as each $H_i$ has finite index, it is also infinite and finitely generated. As $k_1$ is finite, each $L_i$ is also infinite and finitely generated. Thus $H_1 = L_1^{(1)} \times \dots \times L_1^{(k_1)}$ is a subgroup of $G$ on which $\ZZ^2$ acts translation-like. The result follows from~\Cref{corollary_caract_entropies_full}.
\end{proof}

A canonical example which satisfies all of the above properties is the following.

\begin{example}
	The set of topological entropies of non-empty SFTs in the Grigorchuk group~\cite{Grigorchukgrouporiginal1984} is exactly the set of non-negative upper semi-computable numbers.
\end{example}

\section{Final remarks}\label{subsec:finalrem}

The techniques presented in this work give tools to embed the entropies of SFTs defined on a group $G$ to groups in which $G$ embeds geometrically. As the only known non-trivial base cases are $\ZZ$ and $\ZZ^2$, we can only obtain characterizations which coincide either with $\entsft(\ZZ)$ or $\entsft(\ZZ^2)$. This raises the following question.

\begin{question}\label{question-intermediate}
	Is there any infinite and finitely generated amenable group $G$ with decidable word problem for which $\entsft(G)$ is neither $\entsft(\ZZ)$ nor $\entsft(\ZZ^2)$? 
\end{question}

Furthermore,~\Cref{theorem_caract_entropies_G_z2_translation_like} provides a full characterization of the entropies attainable by SFTs defined on polycyclic-by-finite groups, but it cannot be applied on every solvable group with decidable word problem. Two notable examples where it does not apply (at least not directly) are the Baumslag-Solitar groups $\texttt{BS}(1,n) = \langle a,b \mid bab^{-1} = a^n\rangle$ for $n \geq 2$, and the Lamplighter group $ \ZZ/2\ZZ \wr \ZZ$.

\begin{question}\label{question:BS}
	For $n \geq 2$, does it hold that $\entsft(\texttt{BS}(1,n)) = \entsft(\ZZ^2)$? 
\end{question}

\begin{question}\label{question:LL}
	Characterize $\entsft(\ZZ/2\ZZ \wr \ZZ)$. Does it coincide with either $\entsft(\ZZ)$ or $\entsft(\ZZ^2)$? 
\end{question}

\begin{acknowledgements*}
	The author wishes to thank Tom Meyerovitch, Mathieu Sablik and Ville Salo for many fruitful discussions. The author is also grateful to an anonymous referee for their helpful remarks. This research was done while the author was a postdoctoral fellow at the university of British Columbia. It was partially supported by the ANR project CoCoGro (ANR-16-CE40-0005) and the ANR project CODYS (ANR-18-CE40-0007).
\end{acknowledgements*}

\bibliographystyle{plain}
\bibliography{ref.bib}

\appendix
\section{Corrigendum}

The purpose of this corrigendum is to bring attention to an error in the proof of~\Cref{theorem_caract_entropies_G_z2_translation_like} that we do not know how to fix. We begin by restating it as a conjecture.

\begin{conjecture}\label{conjecture}
    Let $G$ be a finitely generated amenable group with decidable word problem which admits
a translation-like action by $\ZZ^2$. The set of entropies attainable by G-subshifts of finite type is the set of non-negative upper semi-computable numbers.
\end{conjecture}

We shall first explain the error in the proof of~\Cref{theorem_caract_entropies_G_z2_translation_like}. Next we shall discuss possible ways to fix the proof of the theorem and state a less general version of the result. After this, we shall show that all consequences of the main theorem stated in the article still hold, with the exception of~\Cref{corollary_caract_entropies_full}. Finally, we shall provide an update of the current state of the art regarding the questions in~\Cref{subsec:finalrem}

The author would like to thank Ville Salo for discovering the mistake in the proof.

\subsection{The error}
Let $G$ be a finitely generated amenable group with decidable word problem on which $\ZZ^2$ acts translation like. As explained in the proof of~\Cref{theorem_caract_entropies_G_z2_translation_like}, we obtain that for every integer $n>0$, there exists an upper semi-computable number $\delta_n \in [0,\frac{1}{n})$ such that \[ \delta_n + \entsft(\ZZ^2) \subset \entsft(G) \subset \entsft(\ZZ^2).  \]

Then, we incorrectly argue that this implies that $\entsft(\ZZ^2)=\entsft(G)$ due to the facts that $0 \in \entsft(G)$ and upper semi-computable numbers are stable under addition. We were implicitly arguing that the second property implies that \[ \delta_n + \entsft(\ZZ^2) = \{ h \in \entsft(\ZZ^2) : h \geq \delta_n \}.  \]

However, this deduction is incorrect. The issue is that while upper semi-computable numbers are indeed closed under addition, they are not closed under inverses (and thus under subtraction). For instance, take $x$ a computable number and let $\delta$ be an upper semi-computable number which is not computable. We claim that $x$ cannot be written as the sum of $\delta$ with another upper semi-computable number $y$. If that were the case, we would have $y$ = $x-\delta$. As $y$ is upper semi-computable, there is an algorithm which produces a sequence $(q_n)_{n \in \NN}$ of rationals such that $\inf_{n \in \NN}q_n = y$. Similarly, as $x$ is computable, there exists a sequence of rationals $(r_n)_{n \in \NN}$ such that $|x-r_n|\leq 2^{-n}$. It follows that $\sup_{n \in \NN}(r_n-q_n) = \delta$ and thus $\delta$ is lower semi-computable, which contradicts the assumption that $\delta$ is not computable.

What we can conclude with those two facts is only that 
\[ \delta_n + \entsft(\ZZ^2) \subset \{ h \in \entsft(\ZZ^2) : h \geq \delta_n \}.  \]

Which is not enough to prove the equality in~\Cref{theorem_caract_entropies_G_z2_translation_like}.

\subsection{Possible fixes and a restricted result}

A natural way to fix the proof of~\Cref{theorem_caract_entropies_G_z2_translation_like} would be to improve~\Cref{theorem_HG} by giving the additional property that the sub $G$-SFT has computable entropy, or even better, zero topological entropy. 

\begin{question}\label{question:computablesubsft}
    Let $G$ be a finitely generated amenable group with decidable word problem and let $X$ be a $G$-SFT. Does there exist a $G$-SFT $Y\subset X$ such that $\htop(G\curvearrowright Y) \leq \htop(G \curvearrowright X)$ and such that $\htop(G\curvearrowright Y)$ is computable?
\end{question}

As mentioned in~\Cref{question_zeroentropy}, it is still open whether for some countable amenable group $G$ there exists a $G$-SFT which does not contain a zero entropy SFT. If the answer to this question is negative, then the answer to~\Cref{question:computablesubsft} is positive and thus~\Cref{conjecture} would hold. We do now know the answer to any of these questions.

However, even if the answer to~\Cref{question:computablesubsft} were negative, in fact the only thing we need to prove~\Cref{conjecture} is the existence of free $G$-chart for $\ZZ^2$ with arbitrarily low computable entropy. We obtain the following restricted result.

\begin{theorem}\label{prop:restricted}
	Let $G$ be a finitely generated amenable group with decidable word problem. Suppose that for every $\varepsilon>0$ there exists a free $G$-chart $(X,\gamma)$ for $\ZZ^2$ such that $\htop(G\curvearrowright X)< \varepsilon$ is computable. Then $\entsft(G)=\entsft(\ZZ^2)$.
\end{theorem}

\begin{proof}
	On the one hand, as $G$ is finitely generated and has decidable word problem, we get by~\Cref{proposition_ECSubshift_has_USC_entropy} that $\entsft(G)\subset\entsft(\ZZ^2)$. On the other hand, using the hypothesis along~\Cref{corollary_realize_entropy} yields a sequence $(\delta_n)_{n \geq 1}$ of computable numbers which converges to zero and such that $\delta_n + \entsft(\ZZ^2)\subset\entsft(G)$. As $\delta_n$ is computable, we now do indeed have that $\delta_n + \entsft(\ZZ^2) = \{ h \in \entsft(\ZZ^2) : h \geq \delta_n \}$. From this and the fact that $0\in \entsft(G)$ we deduce that $\entsft(\ZZ^2)\subset\entsft(G)$.
\end{proof}

Of course, if there exists a free $G$-chart $(X,\gamma)$ for $\ZZ^2$ such that $\htop(G\curvearrowright X)=0$, then the conclusion of~\Cref{prop:restricted} also holds. Two classes of groups $G$ which admit free $G$-charts for $\ZZ^2$ with zero topological entropy are the following:

\begin{example}\label{ex:zeroent_subgroup}
	Let $G$ be a group on which $\ZZ^2$ embeds. Let $\psi \colon \ZZ^2 \to G$ be an injective homomorphism. Take $X =\{0\}$ a singleton and $\gamma \colon \ZZ^2\times \{0\} \to G$ be given by $\gamma(u,0)=\psi(u)$ it follows that $(X,\gamma)$ is a free $G$-chart of $\ZZ^2$. Furthermore, it is clear that $\htop(G\curvearrowright X)=0$.
\end{example}

\begin{example}\label{ex:zeroent_product}
	Let $G=H_1\times H_2$ be a group, where $H_1,H_2$ are infinite and finitely generated. Let $i \in \{1,2\}$. By~\Cref{theorem_ofSeward}, there exists a translation-like action of $\ZZ$ on $H_i$. Set $F_i = \{f \in H : n\cdot h = fh \mbox{ for } n \in \{-1,1\}, h \in H\}$. Construct the $H_i$-subshift of finite type whose elements codify all possible $\ZZ$-cocycles such that the generators move with range $F_i$, that is, the subshift $X_i$ on alphabet $F_i$ with the property that $\gamma_i \colon \ZZ\times  X_i \to H_i$ generated by $\gamma(1,x)=x(g)$ is a $\ZZ$-cocycle. An explicit construction of this subshift can be found in~\cite{Barbieri2017}.

    For each $X_i$, consider its trivial extension $\widehat{X_i}$ to $H_1\times H_2$ where the symbols are constant on the other coordinate. Take $X = \widehat{X}_1\times \widehat{X}_2$ and note that by construction $\htop(G\curvearrowright X) =0$. Notice that $X$ admits naturally a $\ZZ^2$-cocycle $\gamma$ obtained by putting together $\gamma_1$ and $\gamma_2$, hence $(X,\gamma)$ is a $G$-chart for $\ZZ^2$ which is not necessarily free. Consider the $G$-subshift $X'\subset X$ given by all configurations such that the action induced by the cocycle is free, and remark that by the choce of $F_i$, $X'$ is nonempty.

    Next, take any $\ZZ^2$-SFT with zero topological entropy $Y$ such that the shift $\ZZ^2$-action is free, such as the Robinson tiling~\cite{Robinson1971} and consider the $G$-SFT $Z=Y_{\gamma}[X]$. As $Y$ is free, it follows that $Z = Y_{\gamma}[X']$ and thus by~\Cref{theorem_addition_formula} we obtain \[ \htop(G\curvearrowright Z) = \htop( \ZZ^2 \curvearrowright Y)+\htop(G\curvearrowright X') \leq \htop( \ZZ^2 \curvearrowright Y)+\htop(G\curvearrowright X)=0.\]
    
    Hence, if we equip $Z$ with the cocycle induced by $\gamma$, we obtain a free $G$-chart of $\ZZ^2$ with zero topological entropy.
\end{example}

\subsection{Consequences of the restricted theorem}

The only consequence stated in this article that is no longer valid is~\Cref{corollary_caract_entropies_full}, which is a generalized version of~\Cref{conjecture} for countable groups.

The characterization of entropies for virtually polycyclic groups in~\Cref{subsec:poly} still holds, as the proof only applies the main theorem to groups which contain $\ZZ^2$, and there we can apply instead~\Cref{prop:restricted} by virtue of~\Cref{ex:zeroent_subgroup}. Similarly, the results about direct products of finitely generated groups in~\Cref{subsec:products} also hold, as we know that in that case we can find zero-entropy charts for $\ZZ^2$ (\Cref{ex:zeroent_product}). 

In~\Cref{subsec:product_nonfg}, as mentioned before, we have that~\Cref{corollary_caract_entropies_full} is no longer valid. However, we will argue that~\Cref{corollary_caract_entropies_full2} still holds. We recall that this result states that if $G$ is the product of any pair of amenable, countably infinite, non locally finite groups $G_1,G_2$ with decidable word problem, then $\entsft(G)=\entsft(\ZZ^2)$. Indeed, as neither $G_1$ nor $G_2$ is locally finite, we may extract amenable infinite and finitely generated subgroups $H_1\leq G_1$ and $H_2\times G_2$ with decidable word problem. Thus $H_1\times H_2 \leq G$. By~\Cref{corollary_entropy_ofproducts} it follows that $\entsft(H_1\times H_2)=\entsft(\ZZ^2)$, and so $\entsft(\ZZ^2)\subset \entsft(G)$. The other direction follows from~\Cref{proposition_countablegrouphasUSCentropy}.

Finally, the results about branch groups in~\Cref{subsec:branch} still hold: in the proof of~\Cref{theorem_branch_groups} we can instead just invoke~\Cref{corollary_entropy_ofproducts} and thus the result and its corollaries stand.

\subsection{Current state of the questions asked herein}

As mentioned above,~\Cref{question_zeroentropy} about the existence of a $G$-SFT without any zero-entropy sub $G$-SFT is still open and to the author's knowledge there has been no progress in this direction. Similarly,~\Cref{question-intermediate} which asks whether where exists an infinite and finitely generated amenable group with decidable word problem and whose space of SFT entropies lies strictly between $\entsft(\ZZ)$ and $\entsft(\ZZ^2)$ is still open. 

Questions~\ref{question:BS} and~\ref{question:LL} about the space of topological entropies of SFTs in the Baumslag-Solitar groups $B(1,n)$ and the lamplighter group $\ZZ/2\ZZ \wr \ZZ$ have been completely answered by Bartholdi and Salo in~\cite{bartholdi2024shiftslamplightergroup}: in both cases the space of topological entropies of SFTs coincides with the space of non-negative upper semi-computable real numbers.

\end{document}